\documentclass[11pt]{amsart}
\usepackage{amsmath,amsfonts,amscd,amssymb,amsthm,epsfig,euscript}
\usepackage{mathrsfs}
\usepackage{amsxtra}
\usepackage{verbatim,epsf}
\usepackage{framed}
\newtheorem{theorem}{Theorem}

\newtheorem{lemma}{Lemma}
\newtheorem{corollary}{Corollary}
\theoremstyle{definition}
\newtheorem{definition}{Definition}

\theoremstyle{remark}
\newtheorem{remark}{Remark}
\numberwithin{equation}{section}
\newcommand{\field}[1]{\ensuremath{\mathbb{#1}}}
\newcommand{\CC}{\field{C}}

\newcommand{\PP}{\field{P}}

\newcommand{\ZZ}{\field{Z}}

\newcommand{\curly}[1]{\mathscr{#1}}
\newcommand{\cA}{\curly{A}}

\newcommand{\cC}{\curly{C}}

\newcommand{\cF}{\curly{F}}

\newcommand{\cL}{\curly{L}}
\newcommand{\cM}{\curly{M}}
\newcommand{\cN}{\curly{N}}
\newcommand{\cO}{\curly{O}}
\newcommand{\cP}{\curly{P}}

\newcommand{\cS}{\curly{S}}

\newcommand{\cU}{\curly{U}}

\newcommand{\cW}{\curly{W}}

\newcommand{\GL}{\mathrm{GL}}

\newcommand{\OO}{\mathcal{O}}

\DeclareMathOperator{\tr}{tr} 
 
\DeclareMathOperator{\Aut}{Aut}

\DeclareMathOperator{\End}{End}
\DeclareMathOperator{\Par}{Par}
\DeclareMathOperator{\Res}{Res} 
\DeclareMathOperator{\Ad}{Ad}

\usepackage{hyperref}
\hypersetup{colorlinks,citecolor=blue,plainpages=false,hypertexnames=false}
\begin{document}

\title[ Bundle automorphisms over the Riemann sphere]{Remarks on groups of bundle automorphisms over the Riemann sphere}


\author[Meneses]{Claudio Meneses}
\address{\noindent Centro de Investigaci\'on en Matem\'aticas A.C., Jalisco S/N,
 Valenciana, \indent C.P. 36023, Guanajuato, Guanajuato, Mexico}
\curraddr{Mathematisches Seminar, Christian-Albrechts Universit\"at zu Kiel, \indent Ludewig-Meyn-Str. 4, 
24118 Kiel, Germany}
\email{claudio.meneses@cimat.mx} 

\thanks{}

\thanks{}


\date{}

\dedicatory{}
\begin{abstract}
A geometric characterization of the structure of the group of automorphisms of an arbitrary Birkhoff-Grothendieck bundle splitting $\bigoplus_{i=1}^{r} \OO(m_{i})$ over $\CC\PP^{1}$ is provided, in terms of its action on a suitable space of generalized flags in the fibers over a finite subset $S\subset\CC\PP^{1}$. The relevance of such characterization derives from the possibility of constructing geometric models for diverse moduli spaces of stable objects in genus 0, such as parabolic bundles, parabolic Higgs bundles, and logarithmic connections, 
as collections of orbit spaces of parabolic structures and compatible geometric data satisfying a given stability criterion, under the actions of the different splitting types' automorphism groups, that are glued in a concrete fashion. We illustrate an instance of such idea, on the existence of several natural representatives for the induced actions on the corresponding vector spaces of (orbits of) logarithmic connections with residues adapted to a parabolic structure.
\end{abstract}

\maketitle


\tableofcontents

\section{Introduction}\label{sec:intro}

The aim of this work is  to provide a natural geometric interpretation of the automorphism groups of holomorphic vector bundles over the Riemann sphere. 
According to a foundational theorem of Grothendieck  \cite{Groth57}, every bundle $E\to \CC\PP^{1}$ splits as a direct sum of line bundles\footnote{Grothendieck's result holds for holomorphic principal bundles whose structure group is an arbitrary reductive group. For simplicity, we will only consider the case of $\GL(r,\CC)$ (that is, the vector bundle case).} 
\[
E\cong\bigoplus_{i=1}^{r} \OO(m_{i}). 
\]
Since
\[
\End(E) = E^{\vee}\otimes E \cong \End\left(\bigoplus_{i=1}^{r} \OO(m_{i})\right),
\]
it follows that the group $\Aut(E)$, defined in terms of the invertible elements in the algebra of global holomorphic endomorphisms $H^{0}\left(\CC\PP^{1},\End(E)\right)$, admits a parametrization in terms of matrix-valued polynomials of block-triangular form. Such parametrization of $\Aut(E)$ dependens on the choice of affine trivializations for $E$, and any two affine trivializations differ precisely by an element in $\Aut(E)$. 

What motivates this work is the observation that, despite the previous trivial characterizations of bundle automorphisms, such groups in fact posses an interesting geometric structure. The structure manifests in the study of moduli spaces of vector bundles with stable parabolic structures \cite{MS80} when the base Riemann surface is $\CC\PP^{1}$, their natural K\"ahler structures \cite{MenTak14}, and their rich hyper-K\"ahler generalizations, namely, the moduli spaces of stable parabolic Higgs bundles and logarithmic connections. 
In general, the construction of moduli spaces of stable structures for the previous geometric data on a compact Riemann surface $\Sigma$ requires to fix the topology of $E$ (i.e., its degree and rank), together with an additional numerical data necessary to define the notion of parabolic degree and stability -- a set of marked points on $\Sigma$, and a collection of parabolic weights satisfying certain inequalities. The peculiarities of $\CC\PP^{1}$ among all Riemann surfaces imply that a parabolic bundle may be thought of as a splitting type (a choice of local trivializations of $E$, over each of the affine charts $\CC_{0}$ and $\CC_{1}$ of $\CC\PP^{1}$),\footnote{The splitting coefficients of each underlying bundle are only \emph{holomorphic} (i.e., not topological) invariants of $E$. The different admissible splitting types for a given choice of parabolic weights determine moduli spaces' stratifications.} together with an \emph{orbit} of collections of flags in $\CC^{r}$, under the action of $\Aut (E)$, satisfying a suitable stability condition. In a similar way, a stable parabolic Higgs pair is determined by an orbit of parabolic structures and compatible parabolic Higgs fields satisfying a stability condition. Since both the parabolic Higgs fields and the logarithmic connections in $E$ can be determined by their residues along the set of marked points and a certain matrix-valued differential (see section \ref{sec:gauge}), it is enough to study the action of $\Aut(E)$ on the latter. Such a structure provides relatively simple moduli spaces geometric models, exclusive of $\CC\PP^{1}$ as a Riemann surface. 

If $\textrm{rank}(E) = 2$, the technicalities arising from the geometry of group actions are relatively mild, as the groups $\Aut(E)$ are always a semidirect product of abelian groups \cite{Men17, MenSpi17}.  In general, the geometric properties of the action of the groups $\Aut(E)$ on $n$-tuples of flags are an additional subtlety that needs to be resolved, if any sensible geometric construction of such moduli spaces of stable structures is expected in arbitrary rank. 
At the same time, it is precisely the nature of 
those group actions 
what encodes rich and subtle geometric properties that could render an explicit picture on the wall-crossing phenomenon that manifests when variations of parabolic weights are allowed. The description and study of such group actions doesn't seem to exist in the literature. Moreover, general treatments, such as \cite{MS80,BH95,Tha02}, describe the construction of the corresponding moduli spaces with a caveat precisely for genus 0.  A first step to deal with the arbitrary rank case, would be to find a space of generalized flags on which the action of $\Aut(E)$ is \emph{optimal}.

Following the principle of understanding the structure of a group in terms of its actions, we will construct a space, extracted from the splitting coefficients of the vector bundle $E$, on which the group of bundle automorphisms naturally acts, whose points consist of generalized flags on the fibers over a finite set in $\CC\PP^{1}$, and which is a torsor for $\Aut(E)$, up to a minimal residual subgroup. Intuitively, the group action that we will consider can be thought of as a natural generalization of the standard action of the group $\GL(2,\CC)$ on the configuration space of triples of points in $\CC\PP^{1}$ by means of M\"obius transformations, or certain affine actions  of Nagata type \cite{Muk05}.

The work is organized as follows: section \ref{sec:moduli} presents the general scheme under which the moduli spaces of stable parabolic bundles over $\CC\PP^{1}$ can be constructed by means of the simultaneous actions of groups of bundle automorphisms, which serves as motivation for the rest of the work. In section \ref{sec:pre} we set up conventions and prove a couple of preparatory results. Then, theorem \ref{theo:factorization} in section \ref{sec:bundles} provides a factorization of $\Aut(E)$ in terms of special parabolic subgroups, which constitute our basic building blocks. In section \ref{sec:actions} we construct a space of generalized flags in the bundle $E$, and show in theorem \ref{theo:normalization} (our main result) that the induced action of the group $\Aut(E)$ is transitive and free up to a residual subgroup. The latter result implies that a canonical normalization of such flags can always be attained (corollary \ref{cor:normalization}).  The normalization results are applied in section \ref{sec:gauge} to determine the existence of a couple of holomorphic gauges of logarithmic connections adapted to a given parabolic structure, which we have called the Bruhat and Riemann-Hilbert gauges, and which are described, respectively, in corollaries \ref{cor:Bruhat} and \ref{cor:Riemann-Hilbert}. Further details of such gauges are provided for the rank 2 case.

\section{Parabolic structures and moduli space models on $\CC\PP^{1}$}\label{sec:moduli}

The existence and construction of moduli spaces of rank $r \geq 2$ stable parabolic Higgs bundles and logarithmic connections on a compact Riemann surface $\Sigma$ depends on the a priori choice of a set of \emph{parabolic weights}, i.e. a collection of real numbers 
\[
\cW =\left\{0\leq \alpha_{i1}\leq \dots\leq \alpha_{ir} < 1,\quad i=1,\dots,n, \; :\; \sum_{i=1}^{n}\sum_{j=1}^{r}\alpha_{ij}\in\ZZ\right\}.
\]
which is necessary to define the notions of stability and parabolic degree. In precise terms, a rank $r$ parabolic bundle $E_{*}$ of parabolic degree 0 \cite{MS80} consists of an underlying vector bundle $E\to \Sigma$ of rank $r$, such that 
\[
\textrm{par}\deg (E_{*}) := \deg (E) + \sum_{i=1}^{n}\sum_{j=1}^{r}\alpha_{ij} = 0,
\]
together with a collection of descending flags 
\[
E|_{z_{i}} = E_{i1}\supset \dots\supset E_{i s_{i}}\supset \{0\}
\]
over the fibers of a finite subset $S =\{z_{1},\dots,z_{n}\}\subset\Sigma$, whose multiplicity type is subordinate to the parabolic weights' multiplicities of $\cW$, in the sense that for each $1 \leq i \leq n$, the partitions
\[
r = j_{1} + \dots + j_{s_{i}}
\]
(where $j_{k} = \dim E_{i k} - \dim E_{i k+1}$, and letting $E_{i s_{i} + 1} = \{0\}$), are such that
\[
\alpha_{i1} = \alpha_{ij_{1}} < \alpha _{ij_{1} +1}= \alpha_{i j_{1} + j_{2}} < \dots < \alpha_{ij_{1} + \dots + j_{s_{i} - 1} + 1} = \alpha_{i j_{1} + \dots + j_{s_{i}}}.
\]
Parabolic bundle endomorphisms are then defined to be endomorphisms $g\in H^{0}(\Sigma,\End(E))$ preserving the parabolic structure of $E_{*}$, that is, such that for every $1 \leq i \leq n$ and $1\leq j \leq s_{i}$, 
\[
g\left(E_{ij}\right) \subset E_{ij}.
\]
In other words, for every $1 \leq i \leq n$, the restriction $g|_{z_{i}}$ is an element of the Lie algebra $\mathfrak{p}\left(E|_{z_{i}}\right)\subset \End\left(E|_{z_{i}}\right)$ of the parabolic subgroup $\mathrm{P}\left(E|_{z_{i}}\right)\subset \textrm{GL}\left(E|_{z_{i}}\right)$ stabilizing the flag on $E|_{z_{i}}$.

Given a parabolic bundle $E_{*}$, let us denote by $\mathcal{E}$ the structure sheaf of the underlying bundle $E$. The subsheaf $\Par\End (\mathcal{E})\to \End(\mathcal{E})$ of bundle endomorphisms preserving the parabolic structure of $E_{*}$ is locally-free, and hence equivalent to a vector bundle on $\Sigma$, that we will denote by $\Par\End (E_{*})$ (see remark \ref{remark:explicit-bundle} in section \ref{sec:gauge} for details when  $\Sigma = \CC\PP^{1}$).  The relevance of such a vector bundle is manifested in the vector space
\[
H^{1}\left(\Sigma,\Par\End(E_{*})\right)
\]
which models infinitesimal deformations of the parabolic bundle $E_{*}$.
In the same spirit, a parabolic Higgs field $\Phi$ on $E_{*}$ is an element of the vector space
\[
H^{0}\left(\Sigma,\Par\End(E_{*})^{\vee}\otimes K_{\Sigma}\right).
\]
Equivalently, a parabolic Higgs field in $E_{*}$ is an $\End(E)$-valued meromorphic differential on $\Sigma$, holomorphic on $\Sigma\setminus S$,  and with simple poles at every marked point $z_{i}$, whose residues belong to the unipotent radical Lie algebra $\mathfrak{n}\left(E|_{z_{i}}\right)\subset\mathfrak{p}\left(E|_{z_{i}}\right)$.

In turn, a logarithmic connection $\nabla$ adapted to the parabolic structure of $E_{*}$ (\cite{Biq91}; cf. \cite{Simp90,LS13}), is a meromorphic connection on $E$, holomorphic on $\Sigma\setminus S$, and such that for every $1 \leq i \leq n$, there exist a sufficiently small neighborhood $\cU_{i}$ of any point $z_{i} \in \Sigma$ with local coordinate $z-z_{i}$, and a holomorphic trivialization of $E|_{\cU_{i}}$ mapping the flag in $E|_{z_{i}}$ to the standard flag in $\CC^{r}$ with multiplicities $(j_{1},\dots,j_{s_{i}})$, for which the local basis given by
\[
s_{ij} = (z-z_{i})^{-\alpha_{ij}}\mathbf{e}_{j}
\]
is annihilated by $\nabla$,  
\[
\nabla|_{\cU_{i}}\left(s_{ij}\right) = 0,
\]
where $\left\{\mathbf{e}_{j}\right\}$ is the canonical basis in $\CC^{r}$. Equivalently, over such given local trivialization of $E|_{\cU_{i}}$, a logarithmic connection takes the form
\[
\nabla|_{\cU_{i}} = \mathrm{d} + \frac{W_{i}}{z-z_{i}}\mathrm{d}z
\]
where 
\[
W_{i} = \begin{pmatrix}
\alpha_{i1} & & 0\\
& \ddots & \\
0 & &\alpha_{ir}
\end{pmatrix}
\]
(see section \ref{sec:gauge} for an explicit description when $\Sigma = \CC\PP^{1}$). The corresponding local holomorphic trivialization at $\cU_{i}$ of any other logarithmic connection $\nabla'$ adapted to the parabolic structure of $E_{*}$ would give rise to a holomorphic map $E|_{\cU_{i}} \to E|_{\cU_{i}}$ whose restriction to $z_{i}$ belongs to the parabolic group $P\left(E|_{z_{i}}\right)$. It readily follows that 
\[
\nabla' = \nabla +\Phi
\]
for some parabolic Higgs field $\Phi \in H^{0}\left(\Sigma,\Par\End(E_{*})^{\vee}\otimes K_{\Sigma}\right)$. Hence, the space of logarithmic connections adapted to the parabolic structure of $E_{*}$ is an affine space for $H^{0}\left(\Sigma,\Par\End(E_{*})^{\vee}\otimes K_{\Sigma}\right)$.

Every subbundle $F\subset E$ inherits a parabolic structure from $E_{*}$ according to the following rules. At every $z_{i}$, the intersections 
\[
F|_{z_{i}}\cap E_{ik}
\]
are not necessarily distinct. A reindexing of the different induced subspaces gives rise to a generalized flag $F|_{z_{i}} = F_{i1} \supset F_{i2} \supset \dots \supset F_{is'_{i}}\supset\{0\}$. The weight $\alpha'_{ik}$ of the subspace $F_{ik}$ is decreed as
\[
\alpha'_{ik} = \max\{\alpha_{il} \; :\; F_{ik} = F|_{z_{i}}\cap E_{il}\}.
\]
A parabolic Higgs pair $(E_{*},\Phi)$ for which $\textrm{par}\deg(E_{*}) = 0$ is said to be \emph{parabolic stable} (resp. semi-stable) if for every $\Phi$-invariant subbundle $F\subset E$, 
\[
\mathrm{par}\deg(F_{*}) < 0 \qquad (\text{resp. $\leq$}).
\]
In particular, a parabolic bundle $E_{*}$ of parabolic degree 0 is stable if for every subbundle $F\subset E$, $\mathrm{par}\deg(F_{*}) < 0$. Similarly, a logarithmic connection $\nabla$ adapted to $E_{*}$ is \emph{parabolic stable} (resp. semi-stable) \cite{Simp90} if for every subbundle $F\subset E$ preserved by $\nabla$, we have that $\mathrm{par}\deg(F_{*}) < 0$ (resp. $\leq$).

The case $\Sigma = \CC\PP^{1}$ is rather special, as not only $n \geq 3$ necessarily, but also not every admissible set of weights for a fixed (necessarily negative) degree determines nonempty moduli spaces. Jeffrey \cite{Jef94} provided a moment map construction of all moduli spaces of stable parabolic bundles for a given degree and rank, exhibiting the convex polytope nature of the space of admissible weights yielding nonempty moduli spaces. In \cite{Bis98,Bel01,Bis02}, Biswas  and Belkale provided independently an explicit description of such a weight polytope. 
Generically, the admissible weights are complete (i.e., don't have multiplicities), and hence the flags on each fiber can be parametrized by the complete flag manifold $\cF_{r}$, and moreover, every semi-stable parabolic bundle is stable, so that the corresponding moduli space is compact \cite{BH95}. Let us fix once and for all a choice of such $\cW$. We will denote by $\cM$ (resp. $\cL$) the moduli space of stable parabolic Higgs bundles of parabolic degree 0 (resp. logarithmic connections) with respect to $\cW$, and by $\cN\subset\cM$ the corresponding moduli space of stable parabolic bundles. 

The crucial observation is that there exists a relatively simple description of the points in $\cN$ and $\cM$, encoded entirely in the geometric properties of the action of the groups $\Aut\left(E\right)$. The Birkhoff-Grothendieck theorem \cite{Groth57} states that every holomorphic vector bundle $E\to \CC\PP^{1}$ is isomorphic to a direct sum of line bundles $\bigoplus_{i=1}^{r} \OO(m_{i})$, that is, if we cover $\CC\PP^{1}$ with affine charts $\{\CC_{0},\CC_{1}\}$, $\CC_{0}\cap\CC_{1}=\CC^{*}$, $E$ can be prescribed by a single diagonal transition function
\[
g_{01}(z)=
\begin{pmatrix}
z^{m_{1}} & & 0 \\
 & \ddots &  \\
0 & & z^{m_{r}}
\end{pmatrix}
\]
which can be written more succinctly as $g_{01}(z)=z^{N}$ if we let
\[
N = \begin{pmatrix}
m_{1} & & 0\\
& \ddots &\\
0 & & m_{r}
\end{pmatrix}.
\]
Moreover, any two such isomorphisms differ by postcomposition by an automorphism of the bundle splitting. To simplify notations, we will denote such bundle splitting types simply as $E_{N}$. The parabolic stability condition on any given $\left(E_{N}\right)_{*}$ implies that 
\[
H^{0}\left(\CC\PP^{1}, E_{N} \right) = 0,
\]
hence $m_{1},\dots, m_{r}  < 0$, yielding a finite number of potential splitting matrices $N$ for each degree $ -nr < \deg\left(E_{N}\right) \leq -r$.

For each $i=1,\dots, n$, let  $\cF\left(E_{i}\right)$ be the manifold of complete descending flags on the fiber $E_{i} = E|_{z_{i}}$ over $z_{i}$ in $E$. Every isomorphism $E \cong E_{N}$ induces isomorphisms $\cF\left(E_{i}\right) \cong \cF_{r}$. 
Hence, a parabolic structure on $E$ can be modeled as a point in the product of complete flag manifolds 
\[
\cF^{1}_{r}\times\dots\times\cF^{n}_{r}
\]
and two $n$-tuples of flags define equivalent parabolic structures if and only they differ by the left action of an element in $\Aut\left(E_{N}\right)$. Consequently, a point in $\cN$ is \emph{determined uniquely} by a pair 
\[
\left(E_{N},[\mathbf{F}]\right)
\]
where 
$[\mathbf{F}]$ is an {orbit} of $n$-tuples $\mathbf{F} = (F_{1},\dots,F_{n})$ of complete descending flags under the induced action of the group $\Aut\left(E_{N}\right)$, satisfying a stability condition dictated by $\cW$. Similarly, in terms of a pair $\left(E_{N},\mathbf{F}\right)$, a parabolic Higgs field $\Phi$ is uniquely determined by the data $\mathbf{B}$ consisting of a collection of residue matrices 
\[
c_{i}\in \mathfrak{n}\left(F_{i}\right),\qquad 1 \leq i \leq n,
\]
(where $\mathfrak{n}\left(F_{i}\right)$ is the unipotent radical of the Lie algebra of $\mathrm{P}(F_{i})$, the parabolic subgroup stabilizing $F_{i}$) satisfying a compatibility condition, together with an element 
\[
F\in H^{0}\left(\CC\PP^{1}, \End\left(E_{N}\right)\otimes K_{\CC\PP^{1}}\right) 
\]
(corollary \ref{corollary:parabolic Higgs field}). Hence, a point in $\cM$ is determined uniquely by a pair 
\[
\left(E_{N},[(\mathbf{F},\mathbf{B})]\right)
\]
where $[(\mathbf{F},\mathbf{B})]$ is an $\Aut\left(E_{N}\right)$-orbit of pairs of flag and parabolic Higgs field data, satisfying a stability condition. The same idea holds for the moduli spaces $\cL$.
We are then lead to the construction of geometric models for the moduli spaces $\cL$, $\cM$ and $\cN$ in terms of explicit parameter spaces. For simplicity, we will only describe such a construction in the case of $\cN$, as it motivates the main results of this work. The mechanism for the construction of $\cL$ and $\cM$ is analogous.

For each isomorphism class $\{E_{*}\}\in \cN$, the splitting type of the underlying vector bundle $E$ is encoded in its \emph{Harder-Narasimhan filtration}.  Such filtration determines a stratification of $\cN$  
\[
\cN = \bigsqcup \cN_{N},
\]
where each stratum is formed by equivalence classes of stable parabolic bundles over $\CC\PP^{1}$ whose underlying holomorphic bundle has splitting type $N$. 
In particular, over each connected component of a Harder-Narasimhan stratum $\cN_{N}$ we can fix simultaneous isomorphisms $E\cong E_{N}$ in such a way that the isomorphism classes of stable parabolic structures correspond to a space of $\Aut\left(E_{N}\right)$-orbits on $\cF^{1}_{r}\times\dots\times\cF^{n}_{r}$. 
If we denote by $\cO_{N}$ the space of $\Aut\left(E_{N}\right)$-orbits of $n$-tuples of flags in $\cF^{1}_{r}\times\dots\times\cF^{n}_{r}$ that are stable when regarded as parabolic structures in some $E_{N}$, we have that
\[
\cN \cong \bigsqcup_{N} \cO_{N}\bigg/\sim 
\]
where the relation $\sim$ is expressible in terms of the intersection of orbit closures in $\cF^{1}_{r}\times\dots\times\cF^{n}_{r}$.\footnote{It should also be remarked that Loray and Saito \cite{LS13} have studied the natural holomorphic symplectic structure of the moduli spaces $\cL$ of logarithmic connections in the case of rank 2 and odd degree using algebro-geometric techniques, in such a way that the moduli spaces $\cN$ arise as the bases of natural Lagrangian fibrations. In particular, they specialize their construction to a choice of parabolic weights for which $\cN$ contains a single Harder-Narasimhan stratum, allowing them to introduce charts and manifold structure in $\cL$, and to study its birrational geometry in an explicit way.} 

Moreover and most importantly, if variations of $\cW$ in the weight polytope are allowed, the induced transformations of $\cN$ that occur when a semistability wall is hit can be understood in terms of degenerations of families of complete flag manifolds (cf. \cite{BH95}). The proof of such results, and ``wall-crossing behavior", depends on a careful and intricate analysis of the geometric characterization of stability, and will be discussed elsewhere.\footnote{As an illustration, we describe a minimal example of such behavior. Consider the case $r = 2$, $\deg(E) = -4$, and $n = 4$, with arbitrary parabolic weights in the open chamber determined by the eight inequalities
\begin{eqnarray}
\alpha_{12} +\alpha_{i2} +\sum_{j \neq 1,i}\alpha_{j1} &>& 2,\quad i = 2,3,4,\label{eq:ineq1}\\
 \alpha_{i1} + \sum_{j \neq i} \alpha_{j2} &<& 3, \quad i = 1,2,3,4,\label{eq:ineq2}\\
\sum_{i = 1}^{4}\alpha_{i1} &<& 1 \label{eq:ineq3}
\end{eqnarray}
which is non-empty, since for any $0 < \delta < \alpha < 1/4$  we can choose 
\[
\alpha_{11} = \delta,\quad \alpha_{21} = \alpha_{31} = \alpha_{41} = \alpha ,\qquad \alpha_{i2} = 1 - \alpha_{i1}.
\]
In particular, a semi-stable parabolic structure is always  strictly stable. Over the splitting  $E = \mathcal{O}(-2)^{2}$, the flag manifolds $\cF(E_{i})$ get identified with the aid of the subbundles $\mathcal{O}(-2)\hookrightarrow E$. It can be verified that for any weights satisfying \eqref{eq:ineq1}--\eqref{eq:ineq2}, the set $\cP^{s}$ of stable parabolic structures on such splitting is biholomorphic to $\cC_{4}$, a configuration space of 4 points in $\cF_{2} \cong \CC\PP^{1}$. Hence, the stratum $\cN_{N}$ of stable parabolic bundles with splitting type $\mathcal{O}(-2)^{2}$ is obtained as the quotient
\[
\cN_{N} \cong P\left(\Aut\left(\mathcal{O}(-2)^{2}\right)\right)\setminus \cP^{s}  \cong \CC\PP^{1} \setminus \{0,1,\infty\}.
\]
In turn, the splitting $E = \mathcal{O}(-3)\oplus\mathcal{O}(-1)$ admits exactly one stable orbit. The unique subbundle $\mathcal{O}(-1)\hookrightarrow E$ stratifies every $\cF(E|_{z})$ 
\[
\cF(E|_{z}) \cong \PP(E|_{z}) = \CC \sqcup \{\infty\}, 
\]
with the lines $\{\infty\}\in\PP(E|_{z})$ being stabilized by $\Aut(E)|_{z}$. It readily follows from \eqref{eq:ineq2} that any $\Aut(E)$-orbit containing a flag at $\infty$ is automatically unstable. Moreover, the affine subspace 
\[
\CC^{4}\subset \PP(E|_{z_{1}})\times \PP(E|_{z_{2}})\times \PP(E|_{z_{3}})\times \PP(E|_{z_{4}}), 
\]
is decomposed into exactly two $\Aut(E)$-orbits, namely, $\Aut(E)\cdot(0,0,0,0)$ and its complement. It can be verified that such orbits correspond to the orbits of the standard  $\CC^{*}$-action $t\cdot w = tw$ on $\CC$, and as a consequence of \eqref{eq:ineq2}--\eqref{eq:ineq3} only the latter (that is, the one corresponding to $\CC\setminus \{0\}$ in $\CC$) is stable.

The picture that emerges from such considerations is the following. The moduli space $\cN$ can be reconstructed from the open stratum $\cN_{N}\subset \cN$ as a 3-point compactification. A neighborhood $\cU_{i}$ of each of these points $w_{1}, w_{2}, w_{3}$ corresponds to a local holomorphic family of stable parabolic bundles, whose underlying bundle $E_{w}$ is isomorphic to $\mathcal{O}(-2)^{2}$ if $w \neq w_{i}$, and to $\mathcal{O}(-3)\oplus \mathcal{O}(-1)$ otherwise.}

\section{Parabolic groups and Bruhat decompositions for $\GL(r,\CC)$}\label{sec:pre}

We will recall a few standard facts for convenience of the exposition. Further details of the general theory may be found in \cite{Bor69,FH91}. For any $r\geq 2$, let $V$ be a complex $r$-dimensional vector space. For any partition $r=i_{1}+\dots + i_{s}$, which we will denote by $\lambda$, let $F_{\lambda}$ be a generalized flag
\[
V = V_{1} \supset V_{2}\supset\dots\supset V_{s}\supset\{0\}
\]
where $i_{k} = \dim V_{k} - \dim V_{k+1}$. The stabilizer of such flag in $\GL(V)$ is a parabolic subgroup $\mathrm{P}\left(F_{\lambda}\right)\subset \GL(V)$. Conversely, any parabolic subgroup of $\GL(V)$ is the stabilizer of some such flag. Upon the choice of a basis $\{\mathbf{e}_{1},\dots,\mathbf{e}_{r}\}$ of $V$ such that
\begin{equation}\label{eq:basis}
V_{k} = \text{Span}\{\mathbf{e}_{i_{1}+\dots+i_{k - 1}+1},\dots,\mathbf{e}_{r}\},
\end{equation}
the group $\mathrm{P}\left(F_{\lambda}\right)$ gets identified with the subgroup  $\mathrm{P}_{\lambda}\subset\mathrm{GL}(r,\CC)$ consisting of block-lower triangular matrices for which the $jk$-th block (for $j\geq k$) is of size $i_{j}\times i_{k}$. In particular, every choice of a generalized flag $F$ in $V$ determines a homogeneous space model for the flag manifold $\cF_{\lambda}$, as the quotient $\mathrm{GL}(V)/\mathrm{P}\left(F_{\lambda}\right)$.

Every parabolic subgroup $\mathrm{P}\left(F_{\lambda}\right)\subset \mathrm{GL}(V)$ admits a semidirect product decomposition 
\[
\mathrm{P}\left(F_{\lambda}\right) =  \mathrm{N}\left(F_{\lambda}\right) \rtimes \mathrm{D}\left(F_{\lambda}\right),
\] 
where $\mathrm{N}\left(F_{\lambda}\right)$, the \emph{unipotent radical} of $\mathrm{P}\left(F_{\lambda}\right)$, is the normal subgroup generated by unipotent elements acting as the identity on every quotient $V_{k}/V_{k+1}$. The complementary group $\mathrm{D}\left(F_{\lambda}\right)$ is not unique. Any choice of it is called a \emph{Levi complement} for $\mathrm{N}\left(F_{\lambda}\right)$, and corresponds to the stabilizer of a choice of splittings of the exact sequences 
\[
0 \to V_{k+1} \to V_{k} \to V_{k}/V_{k+1} \to 0
\]
although any two Levi complements are conjugate.  $\mathrm{P}\left(F_{\lambda}\right)$ contains a minimal parabolic subgroup $\mathrm{B}\left(F_{\lambda}\right)$, known as a \emph{Borel subgroup}. Under the isomorphism $\mathrm{P}\left(F_{\lambda}\right)\cong \mathrm{P}_{\lambda}$, $\mathrm{N}_{\lambda}$ consists of block-lower triangular matrices which are block-unipotent, that is, whose $j$-th diagonal block is equal to $\mathrm{Id}_{i_{j}\times i_{j}}$. The Levi complement of $\mathrm{N}_{\lambda}$ in $\mathrm{P}_{\lambda}$ can then be chosen to consist of the invertible block-diagonal matrices. Let us denote by $\lambda_{r}$ the partition 
\[
r = \underbrace{1 + \dots + 1}_{\text{$r$ times}}. 
\]
Then $\mathrm{P}_{\lambda_{r}}=\mathrm{B}(r)$ is the Borel group of invertible lower triangular matrices, while $\mathrm{N}_{\lambda_{r}}=\mathrm{N}(r)$ is its subgroup of all unipotent matrices. 

\begin{remark}\label{rem:choices}
The identification of parabolic subgroups in $\mathrm{GL}(V)$ and their Levi decompositions, with the concrete matrix groups above, depended on the choice of a basis  of $V$ satisfying \eqref{eq:basis}. 
We will therefore assume that such choice has been made in $V$, and restrict to the matrix group case henceforth. In subsequent sections, we will emphasize the origin of such choices.
\end{remark}

The structure of the Lie algebra of block-lower diagonal matrices $\mathfrak{n}_{\lambda}=\textrm{Lie}(\mathrm{N}_{\lambda})$ is determined by its lower central series. Concretely, denoting by $\mathfrak{n}_{\lambda,jk}\subset \mathfrak{n}_{\lambda}$ the abelian subalgebra of matrices whose nonzero elements lie in the $jk$-th block, one has that $\left[\mathfrak{n}_{\lambda,jk},\mathfrak{n}_{\lambda,k'l}\right]=\delta_{kk'}\mathfrak{n}_{\lambda,jl}$. In particular, 
\[
\mathfrak{n}_{\lambda, jk}=\left[\dots\left[\mathfrak{n}_{\lambda, j j-1},\mathfrak{n}_{\lambda, j-1 j-2}\right],\dots, \mathfrak{n}_{\lambda, k+1 k}\right].
\]

Let $\mathrm{W}(r)$ denote the Weyl group of $\GL(r,\CC)$, and let $\mathrm{W}_{\lambda}$ be the subgroup given as the Weyl group of $\mathrm{D}_{\lambda}$. Then, $\mathrm{W}(r)$ is isomorphic to the symmetric group $\mathrm{S}_{r}$, while $\mathrm{W}_{\lambda}$ is isomorphic to $\mathrm{S}_{i_{1}}\times\dots\times\mathrm{S}_{i_{s}}$ and both can be explicitly realized as subgroups of $\GL(r,\CC)$ by means of permutation matrices. 

Let us denote by $N_{\lambda,[\Pi]}$ the subgroup $\Ad\left(\Pi\right)\left( \mathrm{P}_{\lambda}\right) \cap \mathrm{N}(r) \subset \mathrm{N}(r)$, for any $[\Pi]\in \mathrm{W}(r)/ \mathrm{W}_{\lambda}$, which is clearly independent of the choice of representative $\Pi$ in $[\Pi]$.

\begin{lemma}\label{lemma:semidirect}
For every class $[\Pi]\in \mathrm{W}(r)/\mathrm{W}_{\lambda}$, there is a factorization 
 \[
 \mathrm{N}(r) = \mathrm{N}_{\lambda,[\Pi]}^{c} \cdot \mathrm{N}_{\lambda,[\Pi]} 
\] 
for a subgroup $\mathrm{N}_{\lambda,[\Pi]}^{c}\subset \mathrm{N}(r)$, in such a way that $\mathrm{N}_{\lambda,[\Pi]}^{c}$ and $\mathrm{N}_{\lambda,[\Pi]}$ intersect trivially: 
\[
N_{\lambda, [\Pi]}^{c} \cap N_{\lambda,[\Pi]} = \{\textrm{Id}_{r\times r}\}.
\]
\end{lemma}
\begin{proof}
Let $i_{0}=0$. The partition $\lambda$ determines a collection of $s$ subsets $\cS_{1},\dots,\cS_{s}\subset\{1,\dots,r\}$, whose elements are the consecutive integers in a given partition interval: $\cS_{j}=\{i_{0}+\dots+i_{j-1}+1,\dots,i_{0}+\dots+i_{j}\}$.
Recall that if $\Pi$ is the permutation matrix corresponding to a permutation $\sigma\in \mathrm{S}_{r}$, that is, $(\Pi)_{jk}=\delta_{\sigma(j)k}$, then the adjoint action of $\Pi$ on a matrix $A=(a_{jk})$ is given as $\Ad(\Pi)(A)_{jk}= a_{\sigma(j)\sigma(k)}$. Then, by definition, the subgroup $\mathrm{N}_{\lambda,[\Pi]}=\Ad\left(\Pi\right)\left(\mathrm{P}_{\lambda}\right) \cap \mathrm{N}(r)$ consists of the unipotent matrices with a zero $jk$-entry, for $ j > k $, whenever $\sigma^{-1}(j)<\sigma^{-1}(k)$, unless $\{\sigma^{-1}(j), \sigma^{-1}(k)\} \subset \cS_{l}$ for some $1 \leq l \leq s$. Define the subset $\mathrm{N}_{\lambda,[\Pi]}^{c}\subset \mathrm{N}(r)$ to consist of those unipotent matrices with a zero $jk$-entry (for $j>k$) whenever $\sigma^{-1}(j)>\sigma^{-1}(k)$, or $\sigma^{-1}(j)<\sigma^{-1}(k)$ if $\{\sigma^{-1}(j), \sigma^{-1}(k)\} \subset \cS_{l}$ for some $1 \leq l \leq s$. Then, in particular, $\mathrm{N}_{\lambda,[\Pi]}^{c} \cap \mathrm{N}_{\lambda,[\Pi]} = \{\textrm{Id}_{r\times r}\}$. It is straightforward to verify that the defining equations of $\mathrm{N}_{\lambda,[\Pi]}^{c}$ are preserved under multiplication, i.e., that $\mathrm{N}_{\lambda,[\Pi]}^{c}$ is a subgroup of $\mathrm{N}(r)$. Moreover, consider any $C \in \mathrm{N}(r)$. That there is a unique solution to the equation $C =AB$, with $A \in \mathrm{N}_{\lambda,[\Pi]}^{c}$, $B \in \mathrm{N}_{\lambda,[\Pi]}$, can be seen inductively. For any $1\leq j \leq r-1$, either $\sigma^{-1}(j) < \sigma^{-1}(j+1)$ or $\sigma^{-1}(j) > \sigma^{-1} (j+1)$ giving a unique solution to $c_{j+1 j} = a_{j+1 j} + b_{j+1 j}$. Now, assuming we know all $a_{l_{1}l_{2}}$ and $b_{l_{1}l_{2}}$ with $l_{1}-l_{2} < j-k$, the same idea determines $a_{jk}$ and $b_{jk}$ from the equation
\[
n_{jk} = a_{jk} + b_{jk} +\sum_{l = k+1}^{j-1} a_{jl}b_{lk}.
\]
\end{proof}

\begin{definition}
For any partition $\lambda$, and any class $[\Pi]\in \mathrm{W}(r)/\mathrm{W}_{\lambda}$, we will let $\mathrm{N}_{\lambda,[\Pi]}^{c}$ be the subgroup of $\mathrm{N}(r)$ defined in lemma \ref{lemma:semidirect}.
\end{definition}

\begin{remark}\label{remark:complement}
Incidentally, the group $\mathrm{N}_{\lambda,[\Pi]}^{c}$ is also of the form $\mathrm{N}_{\lambda_{0},[\Pi_{c}]}$ for some $[\Pi_{c}]\in \mathrm{W}(r)/\mathrm{W}_{\lambda}$, where $\lambda_{0} = i_{s} + \dots + i_{1}$.  Consider the involution on $\mathrm{S}_{r}$, defined by mapping a given permutation $\sigma^{-1}$ to the complementary permutation
\[
\sigma^{-1}_{c}(k)=r-\sigma^{-1}(k)+1.
\] 
Then, $\sigma_{c}$ is the unique permutation in $\mathrm{S}_{r}$ satisfying that $\sigma^{-1}_{c}(j) < \sigma^{-1}_{c}(k)$ if and only if $\sigma^{-1}(j) > \sigma^{-1}(k)$, and $\{\sigma^{-1}_{c}(j),\sigma_{c}^{-1}(k)\}\in \cS^{c}_{s - l +1}$ if and only if $\{\sigma^{-1}(j),\sigma^{-1}(k)\}\in \cS_{l}$ for some $1 \leq l \leq s$, and every $1 \leq j< k \leq r$. If a representative $\Pi$ is chosen in $[\Pi]$, and $\sigma$ is its corresponding permutation, then, a representative $\Pi_{c}$ for $[\Pi_{c}]$ can be constructed in terms of $\sigma_{c}$.  
\end{remark}

\begin{remark}
In general, $\mathrm{N}_{\lambda,[\Pi]}$ would not be normal in $\mathrm{N}(r)$, and following remark \ref{remark:complement}, the same is true for its complement $\mathrm{N}_{\lambda,[\Pi]}^{c}$. As an example where neither subgroup is normal, consider $r=4$, the partition $4=2+2$, and the permutation $\sigma=(123)$.
\end{remark}

\begin{lemma}\label{lemma:factorization}
Let $\lambda$ be a nontrivial partition of $r$. Every matrix $g\in \GL(r,\CC)$ can be factored in the form
\[
g = L\Pi P 
\]
with $L\in \mathrm{N}(r)$, $P\in \mathrm{P}_{\lambda}$, and $\Pi$ in a fixed class $[\Pi]\in \mathrm{W}(r)/ \mathrm{W}_{\lambda}$. The class of factors $[L]$ contains a unique element $L$ in $\mathrm{N}_{\lambda,[\Pi]}^{c}$.
\end{lemma}

\begin{proof}
Existence of factorizations $L\Pi P$ follow from the standard Bruhat decomposition
\[
\GL(r,\CC)=\bigsqcup_{\Pi\in \mathrm{W}(r)} \mathrm{N}(r)\cdot \Pi \cdot \mathrm{B}(r).
\]
If $g=L'\Pi' P'$, then $P P'^{-1}=\Pi^{-1}(L^{-1}L')\Pi'$, and it follows that all nonzero entries of $\Pi^{-1}\Pi'$ are nonzero entries of $P P'^{-1}\in \mathrm{P}_{\lambda}$. Hence $\Pi^{-1}\Pi' \in W^{\lambda}$.
To conclude, observe that if $g = L' \Pi P'$, then $L^{-1}L' =\Pi P\left(P'\right)^{-1} \Pi^{-1}$. Hence $L^{-1}L' \in \mathrm{N}_{\lambda,[\Pi]}$. In particular, since from lemma \ref{lemma:semidirect}, every choice of $L$ would factor as $L' L''$, with $L'\in \mathrm{N}_{\lambda,[\Pi]}^{c}$ and $L'' \in  \mathrm{N}_{\lambda,[\Pi]}$, it readliy follows that such $L'$ is independent of the choice of $L$. Hence, there is a unique representative of $L$ lying in $\mathrm{N}_{\lambda,[\Pi]}^{c}$.
\end{proof}

\begin{remark}\label{rem:Bruhat}
As a consequence of lemmas \ref{lemma:semidirect} and \ref{lemma:factorization}, we conclude the generalized Bruhat decomposition for $\GL(r,\CC)$
\[
\GL(r,\CC)=\bigsqcup_{[\Pi]\in \mathrm{W}(r)/\mathrm{W}_{\lambda}} \mathrm{N}_{\lambda,[\Pi]}^{c}\cdot \Pi \cdot\mathrm{P}_{\lambda}.
\]
Consequently, the generalized flag manifold $\cF_{\lambda} \cong \GL(r,\CC)/\mathrm{P}_{\lambda}$ admits a stratification 
\[
\cF_{\lambda} = \displaystyle\bigsqcup_{[\Pi]\in \mathrm{W}(r)/\mathrm{W}_{\lambda}} \cF_{\lambda,[\Pi]}
\]
The strata $\cF_{\lambda,[\Pi]} \cong \mathrm{N}_{\lambda,[\Pi]}^{c} \cdot \Pi \cdot \mathrm{P}_{\lambda}/\mathrm{P}_{\lambda}$ are the so-called \emph{Bruhat cells}. The largest (open) cell, of dimension $\left(r^{2}-(i_{1}^{2}+\dots+i_{s}^{2})\right)/2$, is given by the class of the permutation 
\[
\Pi_{0} = (r,1)(r-1,2),\dots(\lfloor (r +1)/2\rfloor +1, \lfloor r/2\rfloor). 
\]
In turn, it is easy to see that the group $\mathrm{N}_{\lambda,[\Pi_{0}]}^{c}$ actually coincides with $\mathrm{N}_{\lambda_{0}}$, the unipotent radical of the parabolic subgroup associated to the partition 
\[
\lambda_{0} = i_{s}+\dots + i_{1}.
\]
\end{remark}

\begin{remark}\label{rem:flag-action}
Any given element in $\cF_{\lambda}$ is obtained as $g \cdot F_{\lambda}$, by means of the left $\GL(r,\CC)$-action on $F_{\lambda}$, defined as
\[
g \cdot V_{j} =  g(V_{j}), \qquad g\in\GL(r,\CC),
\]
providing the correspondence with the homogeneous space model of $\cF_{\lambda}$. In particular, $F_{\lambda}$ corresponds to the class of the identity $[\mathrm{Id}]$, equal to its own Bruhat stratum (a 0-cell). Therefore, from lemma \ref{lemma:factorization}, we can detect the stratum $g \cdot F_{\lambda}$ belongs to in terms of the Bruhat decomposition of $g$. 
\end{remark}

If we now fix a class $[\Pi]\in  \mathrm{W}(r)/\mathrm{W}_{\lambda}$ and let the $[\Pi]$-stratum of $\GL(r,\CC)$ act on $F_{\lambda}$, we conclude the following corollary.

\begin{corollary}\label{cor:action}
Each Bruhat cell $\cF_{\lambda,[\Pi]}\subset \cF_{\lambda}$ is a torsor for the group $\mathrm{N}_{\lambda,[\Pi]}^{c}$.
\end{corollary}

\begin{definition}\label{def:Bruhat-coord}
The \emph{Bruhat coordinates} of the Bruhat cell $\cF_{\lambda,[\Pi]}\subset \cF_{\lambda}$, for any given $[\Pi]\in \mathrm{W}(r)/\mathrm{W}_{\lambda}$, are the holomorphic coordinates given by the off-diagonal nonzero entries in the group $\mathrm{N}_{\lambda,[\Pi]}^{c}$. By a slight abuse of notation, these will be denoted by $\left(L, [\Pi]\right)$, or simply by $L$, where $L\in \mathrm{N}_{\lambda,[\Pi]}^{c}$.
\end{definition}

We will denote the complete flag manifold $\cF_{\lambda_{r}}$ by $\cF_{r}$, and similarly for its Bruhat cells $\cF_{r,\Pi}$, $\Pi\in \mathrm{W}(r)$.

\section{Birkhoff-Grothendieck splittings over the Riemann sphere}\label{sec:bundles}

The study of the notion of stability and semi-stability of holomorphic vector bundles $E \to \Sigma$ on compact Riemann surfaces leads naturally to the introduction of the \emph{Harder-Narasimhan filtrations}, i.e., Jordan-H\"older filtrations by subbundles of $E$
\[
 E = E_{1} \supset E_{2} \supset \dots \supset E_{s} \supset 0
\]
such that each quotient $E_{k}/E_{k+1}$ is semi-stable (letting $E_{s+1} = 0$), and moreover
\[
\mu\left(E_{s}/E_{s + 1}\right) > \dots > \mu\left(E_{1}/E_{2}\right),
\] 
where $\mu(E) = \deg(E)/\textrm{rank}(E)$ is the slope of $E$. 
When $\Sigma = \CC\PP^{1}$, a vector bundle is semi-stable if and only if it is a twist of the trivial bundle, and consequently, the Birkhoff-Grothendieck splitting of a vector bundle $E$ is intimately related to its Harder-Narasimhan filtration. In more concrete terms, assume that 
\[
E \cong \mathcal{O}(n_{1})^{i_{1}}\oplus \mathcal{O}(n_{2})^{i_{2}}\oplus \dots\oplus \mathcal{O}(n_{s})^{i_{s}}
\]
in such a way that $n_{1} < \dots < n_{s}$. Then, it can be verified that, although the Birkhoff-Grothendieck splitting of $E$ is not unique, for every $k =1, \dots, s$, there is a \emph{unique} subbundle $E_{k}\subset E$ such that
\[
E_{k} \cong \mathcal{O}(n_{k})^{i_{k}}\oplus \mathcal{O}(n_{2})^{i_{2}}\oplus \dots\oplus \mathcal{O}(n_{s})^{i_{s}}
\]
and consequently $E_{k}/E_{k+1} \cong \mathcal{O}(n_{k})^{i_{k}}$ is semi-stable, with slope equal to $n_{k}$. In particular, a generalized flag $F_{z}$ gets induced on every fiber $E|_{z}$. Its multiplicity type is determined by the partition  $\lambda_{N}$ given as $i_{1} + \dots + i_{s} = r$. We will denote its stabilizing parabolic subgroup by $\mathrm{P}\left(F_{z}\right)\subset \mathrm{GL}\left(E|_{z}\right)$.

The group $\Aut(E)$ admits a normal subgroup $\mathcal{N}\subset \Aut(E)$ consisting as those automorphisms acting as the identity on every quotient $E_{k}/E_{k+1}$, $k = 1,\dots,s$. In analogy to the Levi decompositions of parabolic subgroups of $\mathrm{GL}(V)$, there exists semi-direct product decompositions
\[
\Aut (E) = \mathcal{N} \rtimes \mathcal{D}
\]
where the subgroups $\mathcal{D}$ are defined for every choice of isomorphism $E_{N}$, and correspond to the stabilizers every summand $\mathcal{O}(n_{k})^{i_{k}}$.  As a corollary, we conclude
\begin{corollary}
Every isomorphism $E\cong E_{N}$ determines a splitting of the short exact sequence 
\[
e\to \mathcal{N}\to \Aut(E) \to \Aut(E)/\mathcal{N} \to e
\] 
in terms of the stabilizing subgroup of the splitting $E_{N}$.
Moreover, the set of all Birkhoff-Grothendieck splittings of a vector bundle $E \to \CC\PP^{1}$ is a torsor for the subgroup $\mathcal{N}\subset \Aut(E)$.
\end{corollary}

From the coefficient matrix $N$ and the partition $\lambda_{N}$, we can induce a parabolic subgroup  $\mathrm{P}_{\lambda_{N}}\subset\mathrm{GL}(r,\CC)$ of block lower-triangular matrices as before.  To simplify notations, we will denote such parabolic subgroup by $\mathrm{P}_{N}$, and similarly for $\mathrm{N}_{N}$ and $\mathrm{D}_{N}$. 

Now, if an arbitrary choice of isomorphism $E \cong E_{N}$ is made, the group $\Aut(E)$ can then be explicitly described on an affine trivialization in terms of matrix-valued polynomials. By definition, an automorphism of $E_{N}$ is then equivalent to a pair of holomorphic maps $\{g_{i}:\CC_{i}\to \GL(r,\CC)\}_{i=0,1}$ satisfying 
\[
g_{0}=z^{N}g_{1}z^{-N}\qquad \text{on}\;\; \CC^{*},
\]
thus, it readily follows that an automorphism of $E_{N}$ is fully determined by a matrix-valued polynomial 
\[
g(z)=\sum_{l=0}^{n_{s}-n_{1}}P_{l}z^{l},
\]
with $P_{0}\in \mathrm{P}_{N}$, and for each $l>0$, $P_{l}\in\mathfrak{n}_{N}$ has zero $jk$-blocks whenever $n_{j}-n_{k} < l$. Moreover, it readily follows that for every $z\in\CC\PP^{1}$, the choice of trivialization induces an isomorphism 
\[
\Aut(E)|_{z} \cong \mathrm{P}_{N}
\]
in such a way that the factorizations $\Aut (E)|_{z} =\mathcal{N}|_{z} \rtimes \mathcal{D}|_{z}$ correspond to the Levi decompositions of $\mathrm{P}_{N}$. In particular, for every $z\in\CC\PP^{1}$,
\[
\mathcal{N}|_{z}\cong \mathrm{N}_{N},\qquad\mathcal{D}|_{z}\cong \mathrm{D}_{N}. 
\]
The previous facts can be stated in a more invariant manner, given in the following lemma, whose proof is now straightforward. 

\begin{lemma}\label{lemma:restriction}
 Assume that $s \geq 2$. For every $z\in \CC\PP^{1}$, the group $\Aut (E)|_{z}$ coincides with the parabolic subgroup $\mathrm{P}\left(F_{z}\right) \subset \mathrm{GL}\left(E|_{z}\right)$ stabilizing the flag $F_{z}$, in such a way that $\mathcal{N}|_{z} = \mathrm{N}\left(F_{z}\right)$. Consequently, every choice of a Birkhoff-Grothendieck splitting of $E$ induces simultaneous choices of Levi decompositions 
\[
\Aut(E)|_{z} = \mathrm{P}(F_{z})\cong\mathrm{N}\left(F_{z}\right)\rtimes \mathcal{D}|_{z}.
\]  
\end{lemma}

\begin{remark}
In the special case when $E$ is isomorphic to a twist of the trivial bundle, that is, when $s = 1$ and $r = i_{1}$ is the trivial partition, the previous structures degenerate, and $\Aut(E) \cong \GL(r,\CC)$. Since the latter is a rather special case, we will exclude it from our considerations. We will assume henceforth that $s \geq 2$, i.e., the associated bundle $\End(E)$ \emph{is nontrivial}.
\end{remark}

From now on, we will concentrate exclusively in the case $E = E_{N}$, or said it differently, we will choose once and for all an isomorphism $E \cong E_{N}$. In particular, this determines an isomorphism $\mathcal{D} \cong \mathrm{D}_{N}$.

Let us consider the numbers $d_{l} = n_{l+1}-n_{l}$, for $1 \leq l \leq s-1$. Then, for any $s \geq j > k \geq 1$, the dimension of the subspace $\Aut\left( E_{N}\right)_{jk}\subset \Aut\left( E_{N}\right)$, corresponding to the $jk$-th block in the group $\Aut\left( E_{N}\right)$, may be expressed as 
\begin{eqnarray*}
\dim \Aut\left( E_{N}\right)_{jk} & = & i_{j}i_{k}\left(n_{j} - n_{k} +1\right)\\
& = & i_{j}i_{k}\left(d_{k}+d_{k+1}+\dots+d_{j-1}+1\right)
\end{eqnarray*}

We will require the introduction of a collection of subgroups of the group $\mathrm{N}_{\lambda}$ of a special kind. 
Let us consider, for every $1\leq l \leq s-1$, the $s$-cycle $\sigma_{l}\in\mathrm{S}_{s}$
given as 
\[
\sigma_{l}(k) = \left\{
\begin{array}{lr}
k + l & \text{if}\quad k \leq s - l,\\\\
k - s + l & \text{if}\quad k >  s - l.
\end{array}
\right.
\]
Each $\sigma_{l}$ acts on  the partition $\lambda$, inducing a new partition
\[
\lambda_{l}=i_{\sigma_{l}(1)} + \dots + i_{\sigma_{l}(s)}.
\]
We can also induce a special $r$-cycle $\tau_{l}\in\mathrm{S}_{r}$, defined as 
 \[
\tau_{l}(k) = 
\left\{
\begin{array}{lr}
k + i_{1} + \dots + i_{l} & \text{if}\quad k \leq i_{l + 1} + \dots + i_{s},\\\\
k - \left(i_{l + 1} + \dots + i_{s}\right) & \text{if}\quad k > i_{l + 1} + \dots + i_{s}.
\end{array}
\right.
\]
Denote by $\Pi_{l}$ the corresponding permutation matrix of $\tau_{l}$, and let us consider the associated groups $\mathrm{N}_{\lambda_{l},[\Pi_{l}]}^{c}$ (recall lemma \ref{lemma:semidirect} and remark \ref{remark:complement}). 
The next lemma unmasks the structure of such groups.

\begin{lemma}\label{lemma:abelian}
The groups $\mathrm{N}_{\lambda_{l},[\Pi_{l}]}^{c}$, $1 \leq l \leq s-1$, are the abelian subgroups of $\mathrm{N}_{\lambda}\subset \mathrm{N}(r)$ determined by a single lower-left block of size $(i_{l+1} + \dots + i_{s})\times(i_{1} + \dots + i_{l})$.
\end{lemma}

\begin{proof}
By definition, the group $\mathrm{N}_{\lambda_{l},[\Pi_{l}]}^{c}$ has a zero $jk$-th entry ($j>k$) if $\tau^{-1}_{l}(j)>\tau^{-1}_{l}(k)$, or if $\tau^{-1}_{l}(j) < \tau^{-1}_{l}(k)$ when $\{\tau^{-1}_{l}(j), \tau^{-1}_{l}(k)\}$ belong to the same partition interval in $\lambda_{l}$. It is straightforward  to see from the definition of $\tau_{l}$ (a shift by $i_{1} + \dots + i_{l}$ modulo $r$) that the second possibility never occurs, since $\Ad\left(\Pi_{l}\right)\left(\mathrm{D}_{\lambda_{l}}\right) = \mathrm{D}_{\lambda}$. Assume that a pair $j > k$ moreover satisfies that either $i_{1} < j \leq i_{1}+ \dots + i_{l}$ and $k \leq i_{1} + \dots + i_{l-1}$, or $i_{1}+ \dots + i_{l} < k \leq i_{1} + \dots + i_{s-1}$ and $i_{1} + \dots + i_{l} < j$. The permutation 
$\tau_{l}$ has been defined in such a way that then $\tau^{-1}_{l}(j) > \tau^{-1}_{l}(k)$ in both cases, and all such terms vanish. Moreover, assume that $k \leq  i_{1}+ \dots + i_{l} < j$. Then $\tau^{-1}_{l}(j) < \tau^{-1}_{l}(k)$, and there is no vanishing constraint for any of these entries. Therefore, $\mathrm{N}_{\lambda_{l}, [\Pi_{l}]}^{c}$ is determined precisely by the single  lower-left block whose entries' indices satisfy $k \leq i_{1}+ \dots + i_{l} < j$.

Now, we can work with the block structure of $\mathrm{N}_{\lambda_{l},[\Pi_{l}]}^{c}$, with respect to the partition $\lambda$. Restated this way, the block structure is determined by decreeing the $jk$-th block to be equal to 0 if either $j \leq l$, or $l < k$ (figure \ref{subgroups}). 
Therefore, for a given $1 \leq l \leq s-1$, let us assume that $j > k$ satisfy $j > l \geq k$. If $n$ and $n'$ are arbitrary elements in $\mathrm{N}_{\lambda_{l},[\Pi_{l}]}^{c}$, then the $jk$-th block of $nn'$ satisfies
\[
(n n')_{jk} = n_{jk} + n'_{jk} = (n'n)_{jk},
\]
that is, $\mathrm{N}_{\lambda_{l},[\Pi_{l}]}^{c}$ is abelian. However, the subgroups $\mathrm{N}_{\lambda_{l},[\Pi_{l}]}^{c}$ do not pairwise commute. In general, 
\begin{equation}\label{eq:commutator}
\left[\mathrm{N}_{\lambda_{l_{1}},\left[\Pi_{l_{1}}\right]}^{c}, \mathrm{N}_{\lambda_{l_{2}},\left[\Pi_{l_{2}}\right]}^{c}\right] = \mathrm{N}_{\lambda_{l_{1}},\left[\Pi_{l_{1}}\right]}^{c}\cap \mathrm{N}_{\lambda_{l_{2}},\left[\Pi_{l_{2}}\right]}^{c}\qquad \text{if}\quad l_{1}\neq l_{2}.
\end{equation}
\end{proof}

The structure of the groups of block-unipotent matrices $\mathrm{N}_{\lambda_{l},[\Pi_{l}]}^{c}$ is sketched in figure \ref{subgroups} for the 
values $i=1,2,\dots,s-1$.

\begin{figure}[!htb]
\[
\scriptstyle
\begin{pmatrix}
I_{1} & & & & \\
* & I_{2} & &  0 &  \\
* & 0 & \ddots  & & \\
\vdots & \vdots & & I_{s-1} & \\
* & 0 & \hdots & 0 & I_{s} 
\end{pmatrix},
\begin{pmatrix}
I_{1} & & & & \\
0 & I_{2} & & 0 &\\
* & * & \ddots & & \\
\vdots & \vdots & & I_{s-1} & \\
* & * & \hdots & 0 & I_{s}
\end{pmatrix},
\dots, 
\begin{pmatrix}
I_{1} & & & & \\
0 & I_{2} & &  0 &  \\
0 & 0 & \ddots  & & \\
\vdots & \vdots & & I_{s-1} & \\
* & * & \hdots & * & I_{s} 
\end{pmatrix}
\]
\caption{Block structure of the groups $N_{\lambda_{l},[\Pi_{l}]}^{c}$, where $I_{j}=\textrm{Id}_{i_{j}\times i_{j}}$.}\label{subgroups}
\end{figure}

Recall from remark \ref{rem:Bruhat} that $\mathrm{N}_{\lambda} = \mathrm{N}_{\lambda_{0},[\Pi_{0}]}^{c}$. The next result is the heart of our structural characterization of the group $\Aut \left(E_{N}\right)$ (see remark \ref{rem:geom-action}).

\begin{theorem}[Geometric factorization of automorphisms]\label{theo:factorization}
If $s>2$, the group of bundle automorphisms $\Aut(E_{N})$ can be expressed as a semidirect product 
\begin{equation}\label{eq:semidirect}
\Aut(E_{N}) \cong  \left( \left( \mathrm{G}_{1} \cdot \mathrm{G}_{2}\cdot \dots \cdot \mathrm{G}_{s-1}\right) \rtimes \mathrm{N}_{N}\right) \rtimes \mathrm{D}_{N},
\end{equation}
where for each $1\leq l \leq s-1$, the group $\mathrm{G}_{l}$ is isomorphic to the direct product
\begin{equation}\label{eq:semidirect2}
\underbrace{\mathrm{N}_{\lambda_{l},[\Pi_{l}]}^{c}\times\dots\times \mathrm{N}_{\lambda_{l},[\Pi_{l}]}^{c}}_{\text{$d_{l}$ times}}.
\end{equation}
and correspons to restrictions to the fibers over $d_{l}$ different points in $\CC\PP^{1}$.
In the special case $s=2$, there is an isomorphism
\begin{equation}\label{eq:semidirect s=2}
\Aut(E_{N}) \cong \left( \underbrace{ \mathrm{N}_{N} \times  \dots \times \mathrm{N}_{N}}_{\text{$d_{1} + 1$ times}}\right) \rtimes \mathrm{D}_{N},
\end{equation}
\end{theorem}
\begin{proof}
Assume that $s > 2$, and consider each of the blocks $\Aut\left(E_{N}\right)_{jk}$. Those for which $j=k$ conform the $\mathrm{D}_{N}$-factor in $\Aut\left(E_{N}\right)$. When $j > k$, the blocks have the structure of a vector space of matrix-valued polynomials. Let us consider an arbitrary vector space decomposition
\[
\Aut\left(E_{N}\right)_{jk}=V^{0}_{jk}\oplus\left(\bigoplus_{l=k}^{j-1} V^{l}_{jk}\right)
\]
where the spaces $V^{0}_{jk}$ are blocks with 1-dimensional entries, and constrained to form a group isomorphic to $\mathrm{N}_{N}$ (such choices obviously exist), and for every $l>0$, each block entry of $V^{l}_{jk}$ is a subspace of dimension $d_{l}$.

For each $0 \leq l \leq s-1$, consider the subspace $\mathrm{G}_{l}$ of $\Aut\left(E_{N}\right)$ consisting of the block-unipotent elements whose $jk$-th block belongs to the subspace $V^{l}_{jk}$, 
or is trivial when $l\geq j$ or $l <k$. Then, by construction, each subspace $V^{l}_{jk}$ belongs to a unique $\mathrm{G}_{l}$.  In particular 
\[
\mathrm{G}_{0} \cong \mathrm{N}_{\lambda_{0},[\Pi_{0}]}^{c} = \mathrm{N}_{N}.
\]
By definition, each space $\mathrm{G}_{l}$, $l>0$, has the structure of a vector space of dimension $d_{l}\cdot \dim \mathrm{N}_{\lambda_{l},[\Pi_{l}]}^{c}$, and is in fact isomorphic to the direct product of abelian Lie groups
\[
\underbrace{\mathrm{N}_{\lambda_{l},[\Pi_{l}]}^{c}\times\dots\times \mathrm{N}_{\lambda_{l},[\Pi_{l}]}^{c}}_{\text{$d_{l}$ times}}
\]
following from lemma \ref{lemma:abelian}. Thus, such vector space structure is equivalent to an abelian group structure, in such a way that the previous isomorphism is also an isomorphism of groups. 
It is straightforward to verify that the product $\mathrm{G}_{1}\cdot\dots\cdot \mathrm{G}_{s-1}$ is normal in $\Aut\left(E_{N}\right)$. Therefore, the factorization \eqref{eq:semidirect} of the group $\Aut\left(E_{N}\right)$ readily follows.

When $s = 2$ (for example, when $E_{N}$ has rank 2), 
for the only special permutation matrix $[\Pi_{0}]=[\Pi_{1}]$, we have that $\mathrm{N}_{\lambda_{1},[\Pi_{1}]}^{c} = \mathrm{N}_{N}$. Hence, $\mathrm{G}_{1}\cdot {G}_{0}$ is abelian, and the factorization  \eqref{eq:semidirect s=2} follows as a consequence of 
\eqref{eq:semidirect}.
\end{proof}

\begin{remark}\label{rem:N_0}
The action of $\Aut\left(E_{N}\right)$ on each $\cF\left(E_{i}\right)\cong \cF_{r}$ actually preserves certain unions of cells in the Bruhat stratification of $\cF_{r}$
\[
\cF_{r} = \bigsqcup_{\Pi\in \mathrm{W}(r)} \cF_{\Pi} 
\]
inducing subsequent stratification refinements of the Harder-Narasimhan stratification of $\cN$, which are determined by the Bruhat stratifications of the components of $\cF^{n}_{r}$ on every stratum $\cN_{N}$ for which the underlying splitting $E_{N}$ is not a twist of the trivial bundle (see \ref{subsec:Bruhat}).  This finer stratification contains an open and dense set $\cN_{0}\subset \cN$, consisting of evenly-split bundles (remark \ref{rem:evenly-split}) for which the component representatives $F_{i}$ of the stable parabolic structure $\left([\mathbf{F}],\cW\right)$ belong to the $\Aut\left(E_{N}\right)$-orbit of the largest Bruhat cell in $\cF_{r}$.
\end{remark}

\section{Actions of $\mathrm{Aut}\left(E_{N}\right)$}\label{sec:actions}

The next step in the study of the group $\Aut\left(E_{N}\right)$ is the construction of a space where its action is transitive and free up to the action of the residual Levi complement $\mathrm{D}_{N}$. Such space is then a torsor for the subgroup $\mathcal{N}$.

A motivating example, albeit discarded in the present discussion, is the case of the group $\GL(2,\CC)$ with its standard left action on the configuration space of triples of points $w_{i}\in\CC\PP^{1}$ by M\"obius transformations. It is standard that such action is transitive, and every triple is stabilized by $Z(\GL(2,\CC))$. Now, the group $\GL(2,\CC)$ can be thought of as the group of automorphisms of a trivial rank 2 vector bundle on $\CC\PP^{1}$, while each point $w_{i}$ can be thought of as a flag in $\CC^{2}$. Thus, if flags are considered on the fibers of an arbitrary triple of points $z_{1},z_{2},z_{3}$ in the base, the previous result may be interpreted as a statement on the nature of the action of $\Aut \left(E\right)$ on a space of pairwise-different triples of flags $(w_{1},w_{2},w_{3})$ on the fibers in $E$ over $z_{1},z_{2},z_{3}\in\CC\PP^{1}$.

\begin{lemma}\label{lemma:action-D}
For any $1 \leq l \leq s-1$, the group $\mathrm{D}_{N}\subset \Aut\left(E_{N}\right)$ acts on the stratum $\cF_{\lambda_{l},[\Pi_{l}]}\subset \cF_{\lambda_{l}}$ by means of its adjoint action in $\mathrm{N}_{\lambda_{l},\left[\Pi_{l}\right]}^{c}\subset \mathrm{N}_{N}$. The special flags $F_{\lambda_{l},\left[\Pi_{l}\right]} := \left[\Pi_{l} \right] \in \cF_{\lambda_{l},[\Pi_{l}]}$ (remark \ref{rem:flag-action}) are stabilized by $\mathrm{D}_{N}$.
\end{lemma}
\begin{proof}
It is straightforward to verify that the subgroup $\mathrm{N}_{\lambda_{l},\left[\Pi_{l}\right]}^{c}$ is invariant under the adjoint action of $\mathrm{D}_{N}$ in $\mathrm{N}_{N}$ (recall that  $\mathrm{N}_{\lambda_{l},\left[\Pi_{l}\right]}^{c}\subset \mathrm{N}_{N}$). In the homogeneous space model, a point in  $\cF_{\lambda_{l},[\Pi_{l}]}$ corresponds to an orbit  $[L \Pi_{l} P]$, for $L\in \mathrm{N}_{\lambda_{l},\left[\Pi_{l}\right]}^{c}$ fixed, and $P\in\mathrm{P}_{\lambda_{l}}$ arbitrary. Moreover, by definition, we have that 
\[
\Ad\left(\Pi_{l}\right)\left(\mathrm{D}_{\lambda_{l}}\right)=\mathrm{D}_{\lambda}. 
\]
Therefore, an element $D\in \mathrm{D}_{N} = \mathrm{D}_{\lambda}$ acts on such an orbit as $[L \Pi_{l}P] \mapsto [\Ad(D)(L)\Pi_{l}P]$. It readily follows that $F_{\lambda_{l},\left[\Pi_{l}\right]}$ is stabilized by $\mathrm{D}_{\lambda}$, since its Bruhat coordinates are $L = \mathrm{Id}$.
\end{proof}

\begin{definition}\label{def:flag space}
Given a Birkhoff-Grothedieck splitting $E_{N}\to \CC\PP^{1}$, together with $n_{s} - n_{1} + 1 = d_{1} + \dots + d_{s-1} + 1$ arbitrary points $z_{0}=0,\footnote{The normalization $z_{0}=0$ is not strictly necessary, but implementing it accounts for a simpler proof of theorem \ref{theo:normalization}.} z_{1},\dots, z_{d_{1} + \dots + d_{s-1}}$ in $\CC\PP^{1}$, let
\[
\cC_{N}= \cF_{\lambda_{0},[\Pi_{0}]}^{0}\times\left(\prod_{i=1}^{d_{1}}\cF_{\lambda_{1},[\Pi_{1}]}^{i}\right)\times\dots\times\left(\prod_{i=1}^{d_{s-1}}\cF_{\lambda_{s-1},[\Pi_{s-1}]}^{d_{1}+\dots+d_{s-2}+i}\right),
\]
where each $\cF^{i}_{\lambda_{l},[\Pi_{l}]}$ is the corresponding stratum of the flag manifold $\cF_{\lambda_{l}}^{i}$ associated to the fiber $E_{i}$ over the point $z_{i}$ in  $E_{N}$.
\end{definition}

\begin{remark}\label{rem:geom-action}
It is a crucial fact that each of the Bruhat strata that are present in the space $\cC_{N}$ have an intrinsic interpretation in a holomorphic bundle $E$ isomorphic to $E_{N}$, i.e., they are independent of an arbitrary choice of Birkhoff-Grothendieck splitting for $E$. Concretely, for every point $z\in\CC\PP^{1}$ the stratum $\cF^{z}_{\lambda_{l},\left[ \Pi_{l}\right]} \subset \cF^{z}_{\lambda_{l}}$ inside the manifold of generalized flags of type $\lambda_{l}$ in the fiber $E|_{z}$ is characterized by containing those generalized flags $F'_{z}$ whose $(i_{1} + \dots + i_{l})$-dimensional subspace $F'_{s - l + 1}$ is transversal to the $(i_{l + 1} + \dots + i_{s})$-dimensional space $F_{l + 1}$ of the canonically induced Harder-Narasimhan flag, that is,
\[
F'_{s - l + 1}\cap F_{l + 1} = \{0\}.
\]
Thus, the space $\cC_{N}$ is actually a holomorphic invariant of every holomorphic vector bundle $E \to \CC\PP^{1}$, depending only on the auxiliary and arbitrary choice of base points $z_{0},\dots, z_{d_{1} + \dots + d_{s-1}}$.
\end{remark}

By definition, and implementing the fundamental factorization \eqref{eq:semidirect} from theorem \ref{theo:factorization} for a special collection of localized automorphisms, we may obtain an induced natural left action of the group $\Aut\left(E_{N}\right)$ on $\cC_{N}$. From corollary \ref{cor:action}, we can conclude our main result regarding the structure of the group $\Aut(E_{N})$. 

\begin{theorem}\label{theo:normalization}
The induced left action of $ \left(\mathrm{G}_{1}\cdot\dots\cdot \mathrm{G}_{s-1}\right)\rtimes N_{N} \subset\Aut\left(E_{N}\right)$ on $\cC_{N}$ is free and transitive (in particular, the action is proper).
\end{theorem}

\begin{proof}
Let $d_{0}=0$. Recall from theorem \ref{theo:factorization} that the subgroups $\mathrm{G}_{0} \cong \mathrm{N}_{\lambda_{0},\left[\Pi_{0}\right]}^{c} = \mathrm{N}_{N}$ and $\mathrm{G}_{l}$, $l > 0$, are constructed after a choice of subspaces $V^{l}_{jk}$ is made, where $j>k$ by hypothesis (recall that the blocks with $j=k$ determine the subgroup $\mathrm{D}_{\lambda}$). Choose each $V^{0}_{jk}$ to consist of constant entries. In turn, for the remaining block components, let $V_{jk}^{k}\oplus\dots\oplus V_{jk}^{j-1}$ be spanned by the $jk$-th blocks determined by a collection of Lagrange polynomials of degree $d_{k} + \dots + d_{j-1}$ for the points $z_{d_{0} + \dots + d_{k-1}+1},\dots, z_{d_{0} + \dots + d_{j-1}}$, and with a simple zero at $z_{0}=0$,
and for every $k \leq l \leq j-1$, let $V^{l}_{jk}$ be the subspace spanned by the subcollection of Lagrange polynomials for the points $z_{d_{0} + \dots + d_{l-1}+1},\dots, z_{d_{0} + \dots + d_{l}}$ (assumed to vanish at $z_{0}=0$). If one of these points is $\infty$, the corresponding trivialization of $E_{N}$ must be considered. 
The constraint for $\mathrm{G}_{0}$ is obviously satisfied, and each $\mathrm{G}_{l}$, for $l=1,\dots,s-1$, acts freely and transitively on the factor 
\[
\cC_{l}=\prod_{j=1}^{d_{l}}\cF_{\lambda_{l},[\Pi_{l}]}^{d_{0} + \dots + d_{l-1} +j}.
\]
and similarly, the subgroup $\mathrm{G}_{0} = \mathrm{N}_{N}$ acts freely and transitively over the stratum $\cC_{0} = \cF_{\lambda_{0},\left[\Pi_{0}\right]}^{0}$.

In fact, for any $l >0$, the action of $\mathrm{G}_{l}$ on the remaining factors $\{\cC_{l'}\}_{l'\neq l}$, including the stratum $\cC_{0}$, is trivial. The latter case is easy to see, since by construction $\mathrm{G}_{l}|_{z_{0}} =  \{\mathrm{Id}_{r \times r}\}$. For the remaining cases, we first observe that for all $i = d_{0} + \dots + d_{l' - 1} +1,\dots, d_{0} + \dots + d_{l'}$, in a similar way as before, we have that
\[
\mathrm{G}_{l} \cap \mathrm{G}_{l'}|_{z_{i}} = \{\mathrm{Id}_{r \times r}\},
\]
and the claim is true for the subgroup  $\mathrm{G}_{l} \cap \mathrm{G}_{l'} \subset \mathrm{G}_{l}$. Now, due to its abelian nature, we can decompose the group $\mathrm{G}_{l}$ as a direct product 
\[
\mathrm{G}_{l} = \mathrm{G}_{l} \cap \mathrm{G}_{l'} \times \mathrm{H}_{l l'}, 
\]
where $\mathrm{H}_{l l'}$ has a nonzero $jk$-th block if $l < j \leq l'$ and $1 \leq k \leq l$ in the case $l < l'$, or if $ j > l$ and $ l' < k \leq l$ in the case $l > l'$.  In any case, the group 
\[
\Pi_{l'} \mathrm{H}_{l l'} \Pi_{l'}^{-1}
\]
is a subgroup of $\mathrm{N}_{\lambda_{l'}}\subset \mathrm{P}_{\lambda_{l'}}$, which can be readily seen from the fact that $\tau_{l'}$ is the $r$-cycle defined by shifting forward (modulo $r$) by $i_{1} + \dots + i_{l'}$. Since $\mathrm{H}_{l l'}$ and $\mathrm{G}_{l'}$ commute as a consequence of \eqref{eq:commutator}, it follows that the action of $\mathrm{H}_{l l'}$ on $\cC_{l'}$ is also trivial, since every component of an element in $\cC_{l'}$ is of the form $L \Pi_{l'} \cdot F_{\lambda_{l'}}$, with $L\in \mathrm{G}_{l'}$, and by definition, the special flag $F_{\lambda_{l'}}$ is stabilized by $\mathrm{P}_{\lambda_{l'}}$. Consequently, the subgroup $\mathrm{G}_{1}\cdot\dots\cdot \mathrm{G}_{s-1}$ acts freely and transitively on $\cC_{1}\times\dots\times \cC_{s-1}$.

Clearly, the stabilizer in $\left(\mathrm{G}_{1}\cdot\dots\cdot \mathrm{G}_{s-1}\right) \rtimes \mathrm{N}_{N}$ of any point in $\cC_{N}$ is the identity. Now, although the group $\mathrm{N}_{N} = \mathrm{N}_{\lambda_{0},\left[\Pi_{0}\right]}^{c}$ doesn't act trivially on any factor $\cC_{l}$, $l=0,\dots, s-1$, the theorem follows, since given any point in $\cC_{N}$, we can first map its component in $\cC_{0}$ to any given flag in the stratum over $z_{0}$. Then $\mathrm{G}_{1}\cdot\dots\cdot \mathrm{G}_{s-1}$ stabilizes such flag, and also acts freely and transitively in the complement $\cC_{N}\setminus \cC_{0}$. Consequently, the components in $\cC_{N}\setminus \cC_{0}$ can also be normalized in any desired way. Hence, any two points in $\cC_{N}$ can be connected by a unique element in $\left(\mathrm{G}_{1}\cdot\dots\cdot \mathrm{G}_{s-1}\right) \rtimes \mathrm{N}_{N}$. 
This completes the proof.
\end{proof}

Let us consider the special flag representatives $F_{\lambda_{l},\left[\Pi_{l}\right]} = [\Pi_{l}]$ in the Bruhat cells $\cF_{\lambda_{l},\left[\Pi_{l}\right]}$ (see remark \ref{rem:flag-action}). As a consequence of lemma \ref{lemma:action-D} and theorem \ref{theo:factorization}, we conclude the following corollary.

\begin{corollary}[Normalization of flags]\label{cor:normalization}
Under the action of $\Aut\left(E_{N}\right)$, every element in $\cC_{N}$ can be normalized to the special element 
\[
\left(F_{\lambda_{0},[\Pi_{0}]}, \underbrace{F_{\lambda_{1},[\Pi_{1}]},\dots,F_{\lambda_{1},[\Pi_{1}]}}_{\text{$d_{1}$ times}},\dots, \underbrace{F_{\lambda_{s-1},[\Pi_{s-1}]},\dots,F_{\lambda_{s-1},[\Pi_{s-1}]}}_{\text{$d_{s-1}$ times}}\right).
\]
which is stabilized by the subgroup $\mathrm{D}_{N}$.
\end{corollary}

\begin{remark}\label{rem:evenly-split}
The following case is of special importance when holomorphic families of parabolic bundles on $\CC\PP^{1}$ are considered, due to its genericity. A vector bundle $E_{N}\to\CC\PP^{1}$ is said to be \emph{evenly-split} if its splitting coefficients satisfy $|m_{j}-m_{k}|\leq 1$ $\forall\, j,k$. If $E_{N}$ is not a twist of the trivial bundle, then $r=i_{1} + i_{2}$, and $n_{2}-n_{1} = d_{1} = 1$. The evenly-split condition is open on any algebraic family of vector bundles over $\CC\PP^{1}$ containing an evenly-split bundle, and in particular, for the families of underlying bundles on any moduli space of stable parabolic bundles $\cN$. It readily follows that
\[
\Aut\left(E_{N}\right) \cong \left(\mathrm{N}_{N}\times \mathrm{N}_{N}\right) \rtimes \mathrm{D}_{N}.
\]
Upon consideration of the generalized flag manifolds $\cF_{N_{0}}^{0}$ and $\cF_{N_{0}}^{\infty}$, whose points parametrize, respectively, generalized flags of type $r = i_{2} + i_{1}$ in the fibers over the points $z_{0}=0$ and $z_{1}=\infty$ in $\CC\PP^{1}$, our conclusion is that $\Aut\left(E_{N}\right)$ acts transitively on the product of the large Bruhat cells $\cF^{0}_{N_{0},[\Pi_{0}]}\times \cF^{\infty}_{N_{0},[\Pi_{0}]}$. Consequently, every element can be normalized to 
\[
\left(F_{N_{0},[\Pi_{0}]},F_{N_{0},[\Pi_{0}]}\right),
\]
which has the subgroup $\mathrm{D}_{N}$ as its stabilizer. 
\end{remark}

\section{Holomorphic gauge-fixing of logarithmic connections}\label{sec:gauge}

The contents of this section rely on the infinitesimal description of the moduli space $\cN$, for an implicit choice of generic parabolic weights $\cW$ for which $\cN\neq \emptyset$ and all flags are complete. For convenience, we will assume without loss of generality that $z_{n-1} = 0$ and $z_{n} = \infty$. 

It follows from the infinitesimal deformation theory and Serre duality that, given any isomorphism class of stable parabolic bundles $\{E_{*}\}\in \cN$, together with the isomorphism $\{E_{*}\} \cong \left(E_{N},[\mathbf{F}]\right)$, its holomorphic cotangent space can be described, in terms of any given representative $\mathbf{F}\in[\mathbf{F}]$, as the space
\[
T^{*}_{\left(E_{N},[\mathbf{F}]\right)}\cN\cong H^{0}\left(\CC\PP^{1},\Par\End\left(E_{N},\mathbf{F}\right)^{\vee}\otimes K_{\CC\PP^{1}}\right)
\]
where $\Par\End\left(E_{N},\mathbf{F}\right)$ is the holomorphic vector bundle over $\CC\PP^{1}$, of degree 
\[
\deg \left(\Par\End\left(E_{N},\mathbf{F}\right)\right) = -n\dim \cF_{r}, 
\]
associated to the locally-free sheaf of endomorphisms of $E_{N}$ that preserve the quasi parabolic structure determined by $\mathbf{F}$ \cite{MS80,BH95}. Therefore, the possible values that any local section can attain at the points $z_{1},\dots,z_{n}$ form a solvable Lie algebra isomorphic to $\mathfrak{b}(r)$. Moreover, there is a bundle map
\[
\Par\End\left(E_{N},\mathbf{F}\right)\to \End\left(E_{N}\right)
\]
which is an isomorphism away from the fibers over the points $z_{1},\dots,z_{n}$. It follows from the parabolic stability of the pair $\left(E_{N},\mathbf{F}\right)$ with respect to $\cW$ that 
\[
H^{0}\left(\CC\PP^{1},\Par\End\left(E_{N},\mathbf{F}\right)\right) \cong \CC,
\]
with global sections corresponding to the multiples of the identity in $\mathfrak{gl}(r,\CC)$. In particular, the Riemann-Roch theorem for vector bundles then implies that 
\[
\dim\cN = n\dim \cF_{r} -(r^2-1). 
\]
\begin{remark}\label{remark:explicit-bundle}
For the sake of clarity of the future constructions, we will provide a more detailed account of the bundle $\Par\End\left(E_{N},\mathbf{F}\right)$. Recall from definition \ref{def:Bruhat-coord}  that, for any $\Pi\in W(r)$ and its corresponding Bruhat cell $\cF_{r,\Pi} \subset \cF_{r}$, its Bruhat coordinates, corresponding to the nonzero off-diagonal entries of $L \in \mathrm{N}_{r,\Pi}^{c}$, are a consequence of the uniqueness of the standard Bruhat decomposition $g = L\Pi P$ of any $g\in\GL(r,\CC)$ (lemma \ref{lemma:factorization}), thought of as an ordered frame giving rise to a decreasing flag
\[
F = \{V_{1} = \CC^{r}\supset V_{2}\supset\dots,\supset V_{r}\supset \{0\}\},
\]  
and where $P\in B(r)$ represents the arbitrariness of the parametrization. It readily follows that the Lie algebra $\mathfrak{p}(F)$ of endomorphisms of $\CC^{r}$ preserving $F$ satisfies
\[
\mathfrak{p}(F) = \Ad\left(L\Pi\right)\left(\mathfrak{b}(r)\right).
\]
Now, for every $i = 1,\dots,n$, let $\left(\Pi_{i},L_{i}\right)$ be the Bruhat coordinates of the complete descending flag in the $i$th component of $\mathbf{F}$. 
The vector bundle of parabolic endomorphisms $\Par\End\left(E_{N},\mathbf{F}\right)$ can be explicitly constructed as a ``twist" of the endomorphism bundle $\End\left(E_{N}\right)$, if one refines the transition functions of the latter at a collection of non-intersecting punctured disks $\cU_{i}$ centered at each point $z_{1},\dots,z_{n}$. Letting $\cU_{0} = \CC\PP^{1}\setminus S$, and thanks to the algebra structure of $\mathfrak{gl}(r,\CC)$, the transition functions defining $\Par\End\left(E_{N},\mathbf{F}\right)$ can be represented in matrix form as $g_{0i}(z) = \Ad\left(L_{i}\Pi_{i}\right)\left(f_{0i}(z)\right)$, where for $i = 1, \dots, n-1$,
\[
\left(f_{0i}(z)\right)_{kl} = \left\{
\begin{array}{cl}
1 & \text{if $k \geq l$}\\\\ 
(z-z_{i})^{-1} & \text{if $k <  l$},
\end{array}
\right. 
\]
while
\[
\left(f_{0n}(z)\right)_{kl} = \left\{
\begin{array}{cl}
z^{m_{k} - m_{l}} & \text{if $k \geq l$}\\\\ 
z^{m_{k} - m_{l} - 1} & \text{if $k <  l$}.
\end{array}
\right. 
\]
Finally, if any other representative $\mathbf{F}'$ is chosen in the $\Aut\left(E_{N}\right)$-orbit $[\mathbf{F}]$ of $n$-tuples of flags, there is a an induced bundle isomorphism between $\Par\End\left(E_{N},\mathbf{F}\right)$ and the corresponding bundle of parabolic endomorphisms $\Par\End\left(E_{N},\mathbf{F}'\right)$, and similarly for its spaces $H^{k}\left(\CC\PP^{1},\Par\End\left(E_{N},\mathbf{F}'\right)\right)$, $k=0,1$, given explicitly in terms of the element $g\in\Aut\left(E_{N}\right)$ such that $\mathbf{F}' = g\cdot \mathbf{F}$. 
\end{remark}

\begin{lemma}\label{lemma:matrix-differentials}
The elements of $H^{0}\left(\CC\PP^{1},\Par\End\left(E_{N},\mathbf{F}\right)^{\vee}\otimes K_{\CC\PP^{1}}\right)$, i.e., the parabolic Higgs fields for the pair $\left(E_{N},\mathbf{F}\right)$, correspond to the $\mathfrak{gl}(r,\CC)$-valued meromorphic differentials $\Phi (z)$ on $\CC\PP^{1}$, that are holomorphic over $\CC\PP^{1}\setminus S$, and such that, moreover,
\begin{itemize}
\item[(i)]
 $\Phi(z)$ has simple poles over $z_{1},\dots, z_{n-1}\subset\CC$,
 \item[(ii)] $z^{-N}\Phi(z)z^{N}$ has a simple pole at $z_{n} = \infty$,
 \item[(iii)] the residues of $\Phi$ at $z_{1},\dots, z_{n - 1}$ (and of $z^{-N}\Phi z^{N}$ at $z_{n} = \infty$) belong to the nilpotent radical Lie algebras $\mathfrak{n}(F_{i})$, $i = 1,\dots,n$.  
\end{itemize}
\end{lemma}
\begin{proof}
The claims (i)--(iii) readily follow from the construction of the bundle $\Par\End\left(E_{N},\mathbf{F}\right)$ and its dual, described in terms of affine trivializations and transition functions. If $g_{\alpha\beta}$ is a set of transition functions for a bundle $E$, then its dual $E^{\vee}$ is determined by the transition functions $g^{\vee}_{\alpha\beta} = {}^{t}g_{\alpha\beta}^{-1}$. The additional twists over each of the points $z_{1},\dots,z_{n}$ determine the polar behavior of the local sections of $\Par\End\left(E_{N},\mathbf{F}\right)^{\vee}\otimes K_{\CC\PP^{1}}$ over $S$, and its relation to the quasi parabolic structure $\mathbf{F}$.  It is readily seen that the residue at each $z_{i}$ is restricted to lie in the Lie algebra $\mathfrak{n}(F_{i})$. 
\end{proof}

Given a parabolic bundle pair $(E_{N},\mathbf{F})$, let $\Res\left(\mathbf{F}\right)$ be the set of parabolic Higgs fields admitting the local form
\[
\left(\sum_{i = 1}^{n - 1}\frac{c_{i}}{z - z_{i}}\right)\mathrm{d}z
\]
 on the affine trivialization over $\CC_{0}$, where $c_{i}\in \mathfrak{n}(F_{i})$. It readily follows that $\Res\left(\mathbf{F}\right)$ is a subspace of the vector space of parabolic Higgs fields for the pair $(E_{N},\mathbf{F})$, and its elements are completely characterized by the residue values on $z_{1},\dots,z_{n}$.
As a corollary of lemma \ref{lemma:matrix-differentials}, we conclude that every parabolic Higgs field $\Phi$ on the parabolic bundle $(E,\mathbf{F})$ can be split into two pieces of geometric data, 
\[
\Phi = \Res(\Phi) + F
\]
where $\Res(\Phi)$ is the projection of $\Phi$ into $\Res(\mathbf{F})$, while the complementary data $F$ corresponds to an $\End\left(E_{N}\right)$-valued holomorphic differential $F$ on $\CC\PP^{1}$. The induced action of $\Aut \left(E_{N}\right)$ in both subspaces is of adjoint type (lemma \ref{lemma:restriction}). Its orbits are the relevant geometric objects in the moduli theory.

\begin{corollary}\label{corollary:parabolic Higgs field}
For every parabolic bundle pair $\left(E_{N},\mathbf{F}\right)$, its space of parabolic Higgs fields 
\[
H^{0}\left(\CC\PP^{1},\Par\End\left(E_{N},\mathbf{F}\right)^{\vee}\otimes K_{\CC\PP^{1}}\right)
\]
splits as an $\Aut\left(E_{N}\right)$-invariant  direct sum 
\[
H^{0}\left(\CC\PP^{1},\End\left(E_{N}\right)\otimes K_{\CC\PP^{1}}\right)\oplus \textrm{\emph{Res}}(\mathbf{F}).
\]
\end{corollary}

\begin{definition}
A \emph{logarithmic connection adapted to a parabolic structure} $\left(\mathbf{F},\cW\right)$ in the bundle splitting $E_{N}$ (cf. \cite{Simp90,LS13}) is a singular connection, holomorphic over $E_{N}|_{\CC\PP^{1}\setminus S}$, and prescribed in the affine trivialization over $\CC\subset\CC\PP^{1}$ by a meromorphic $\mathfrak{gl}(r,\CC)$-valued 1-form
\begin{equation}\label{eq:germ-0}	 
A(z) = \left(\sum_{i=1}^{n-1}\frac{A_{i}}{z-z_{i}} + f_{0}(z)\right) \mathrm{d}z
\end{equation} 
with $f_{0}(z)$ holomorphic, and such that, in the local coordinate at $\infty$,
\begin{equation}\label{eq:germ-infty}
z^{-N}A(z)z^{N} + N \mathrm{d}z/z
\end{equation}
has a simple pole, whose residue we will denote by $A_{n}$, and for each $i=1,\dots, n$, $A_{i}$ is semisimple with eigenvalues $\{\alpha_{ij}\}_{j=1}^{r}$, and with corresponding ordered eigenlines spanning the $n$-tuple of complete descending flags $\mathbf{F} = \left(F_{1},\dots, F_{n}\right)$.
\end{definition}
In other words, there exists a factorization 
\[
A_{i} = B_{i}W_{i}B_{i}^{-1}
\]
in such a way that, for each $i=1, \dots, n$, $B_{i}$ determines a projective ordered frame giving rise to $F_{i}$. It is not immediate that any given triple $\left(E_{N},\mathbf{F},\cW\right)$ admits addapted logarithmic connections. However, in the special case when such triple determines a stable parabolic bundle, the existence of  a logarithmic connection with irreducible unitary monodromy is guaranteed by the Mehta-Seshadri theorem \cite{MS80}.

\begin{corollary}\label{cor:affine}
For any triple $\left(E_{N},\mathbf{F},\cW\right)$, if nonempty, its space $\cA_{\mathbf{F}}$ of adapted logarithmic connections is an affine space for 
\[
H^{0}\left(\CC\PP^{1},\Par\End\left(E_{N},\mathbf{F}\right)^{\vee}\otimes K_{\CC\PP^{1}}\right).
\]
\end{corollary}
\begin{proof}
This is so since there exist a unique factorization $A_{i} = B_{i}W_{i}B_{i}^{-1}$ for which the Bruhat decomposition $B_{i} = L_{i}\Pi_{i} P_{i}$ is such that $P_{i}\in N(r)$. There is a bijective correspondence between such $P_{i}$ and $c_{i}\in \mathfrak{n}(r)$ such that the following relation holds
\[
P_{i}W_{i}P_{i}^{-1} = W_{i} + c_{i}.
\]
It is then a straightforward consequence of lemma \ref{lemma:matrix-differentials} that the difference of any two logarithmic connections adapted to the same parabolic structure is precisely an element in $H^{0}\left(\CC\PP^{1},\Par\End\left(E_{N},\mathbf{F}\right)^{\vee}\otimes K_{\CC\PP^{1}}\right)$. 
 \end{proof}

In particular, it follows that for every $\left(E_{N},[\mathbf{F}]\right)\in \cN$, there is a bundle $\cA_{[\mathbf{F}]}$ of affine spaces over the $\Aut\left(E_{N}\right)$-orbit $[\mathbf{F}]$, whose fiber over an orbit representative $\mathbf{F}\in[\mathbf{F}]$, corresponds to the space of logarithmic connections $\cA_{\mathbf{F}}$, and which is an affine bundle for the vector bundle over $[\mathbf{F}]$ with fibers $H^{0}\left(\CC\PP^{1},\Par\End\left(E_{N},\mathbf{F}\right)^{\vee}\otimes K_{\CC\PP^{1}}\right)$.  The key point of our construction is the existence of an equivariant action of the group $\Aut\left(E_{N}\right)$ on the bundle $\cA_{[\mathbf{F}]} \to [\mathbf{F}]$, which is given for every  $g\in\Aut\left(E_{N}\right)$ as 
\[
A \mapsto g\cdot A = g A g^{-1} -\mathrm{d}g g^{-1}\in\cA_{\mathbf{F}'},\qquad  A\in\cA_{\mathbf{F}},  \quad \mathbf{F'} = g\cdot \mathbf{F}. 
\]
In a similar way, there is an equivariant action in the vector bundle over $[\mathbf{F}]$, mapping any parabolic Higgs field $\Phi \in H^{0}\left(\CC\PP^{1},\Par\End\left(E_{N},\mathbf{F}\right)^{\vee}\otimes K_{\CC\PP^{1}}\right)$ to 
\[
g\cdot \Phi = g\Phi g^{-1} \in H^{0}\left(\CC\PP^{1},\Par\End\left(E_{N},\mathbf{F'}\right)^{\vee}\otimes K_{\CC\PP^{1}}\right)
\]
inducing a vector space structure in its space of orbits, which we will denote by $\left[H^{0}\left(\CC\PP^{1},\Par\End\left(E_{N},\mathbf{F}\right)^{\vee}\otimes K_{\CC\PP^{1}}\right)\right]$. The latter serves as a model for the holomorphic cotangent space $T^{*}_{\left(E_{N},[\mathbf{F}]\right)}\cN$. 

Consequently, we can speak of an affine space $\left[\cA_{\mathbf{F}}\right]$, whose elements are $\Aut\left(E_{N}\right)$-orbits in $\cA_{[\mathbf{F}]}$, and whose underlying vector space is the orbit space $\left[H^{0}\left(\CC\PP^{1},\Par\End\left(E_{N},\mathbf{F}\right)^{\vee}\otimes K_{\CC\PP^{1}}\right)\right]$. Therefore, there is also an induced affine bundle on $\cN$ for the holomorphic cotangent bundle $T^{*}\cN$, which we will denote simply by $\cA$.


In conclusion, the presence of a Lie group of automorphisms in every bundle splitting $E_{N}$ implies that, for every point $\left(E_{N},[\mathbf{F}]\right)\in\cN$, the logarithmic connections adapted to a pair $\left([\mathbf{F}],\cW\right)$ make sense as an orbit under an action of the group $\Aut\left(E_{N}\right)$. In practice, it would be important to single out orbit representatives whose features are ``as best as possible" in a suitable sense. With such motivation, we formulate the following definition.

\begin{definition}\label{def:gauge}
Let $\left[A_{\mathbf{F}}\right] \in \left[\cA_{\mathbf{F}}\right]$ be an $\Aut\left(E_{N}\right)$-orbit of logarithmic connections on $E_{N}$ adapted to a pair $\left([\mathbf{F}],\cW\right)$. A \emph{holomorphic gauge} for $\left[A_{\mathbf{F}}\right]$  is a choice of representatives
\[
\mathbf{F}\in[\mathbf{F}], \qquad A_{\mathbf{F}}\in\cA_{\mathbf{F}}.
\]
\end{definition}

\begin{remark}\label{rem:centralizer}
The objective at present is to provide a canonical normalization of a logarithmic connection adapted to a parabolic structure on $E_{N}$, in terms of the normalization of the given parabolic structure, which follows from corollary \ref{cor:normalization}. We emphasize that the determination of all the possible splitting types supporting stable parabolic structures is a delicate question at the heart of the general scheme that we propose to construct the moduli spaces $\cN$. Therefore, we make the \emph{a priori} assumption that the pair $\left(E_{N},\mathbf{F}\right)$, not necessarily corresponding to a stable parabolic bundle, is such that 
\[
d_{1} + \dots + d_{s-1} + 1 \leq n
\]
In such case, it readily follows that $\dim \left(\mathrm{G}_{1}\cdot\dots\cdot \mathrm{G}_{s-1}\right)\rtimes \mathrm{N}_{N} \leq n\dim \cF_{r}$. 
The action of $\Aut\left(E_{N}\right)$ on any affine bundle $\cA_{[\mathbf{F}]}$ is never effective; any logarithmic connection adapted to an arbitrary pair $\left(E_{N},\mathbf{F}\right)$ is stabilized by $Z\left(\Aut\left(E_{N}\right)\right)$, which is isomorphic to $\CC^{*}$ since $\mathrm{B}(r)\hookrightarrow \Aut\left(E_{N}\right)$ always, and corresponds to nonzero multiples of the identity.
For every $n$-tuple of flags $\mathbf{F}$, there is an associated subgroup $\mathrm{Stab}(\mathbf{F})$, and moreover, if a logarithmic connection $A_{\mathbf{F}}$ adapted to $\left(\mathbf{F},\cW\right)$ is chosen in $\cA_{\mathbf{F}}$, there is also an induced subgroup $\mathrm{Stab}\left(A_{\mathbf{F}}\right))$, in such a way that
\[
Z\left(\Aut\left(E_{N}\right)\right) \subset \mathrm{Stab}\left(A_{\mathbf{F}}\right)\subset \mathrm{Stab}(\mathbf{F}).
\]
Several important possibilities may occur. 
The most relevant case corresponds to the equality $ \mathrm{Stab}(\mathbf{F}) = \mathrm{Stab}\left(A_{\mathbf{F}}\right) = Z\left(\Aut\left(E_{N}\right)\right)$. In such case, we conclude the following key observation: \emph{given any $\Aut\left(E_{N}\right)$-orbit $\left[A_{\mathbf{F}}\right]\in \left[\cA_{\mathbf{F}}\right]$, there is a unique representative $A_{\mathbf{F}}\in\cA_{\mathbf{F}}$, once an orbit representative $\mathbf{F}\in[\mathbf{F}]$ is chosen.} 
\end{remark}
 
The next and final goal of this work is to study the existence of different natural holomorphic gauges of logarithmic connections for the affine spaces of orbits $\left[\cA_{\mathbf{F}}\right]$.  

\subsection{The Bruhat gauge} \label{subsec:Bruhat}
Let us now consider the  fibration 
\[
\mathrm{pr}_{N} : \cF_{r} \to \cF_{N}.
\]
It follows from lemma \ref{lemma:factorization} and corollary \ref{cor:action} that $\mathrm{pr}_{N}$ is stratification-preserving, i.e.,
the image of every Bruhat cell  $\cF_{r,\Pi}\subset \cF_{r}$ lies inside the Bruhat cell $\cF_{N,[\Pi]}\subset\cF_{N}$, where $[\Pi]\in W(r)/ W_{N}$, and, in fact, $\mathrm{pr}^{-1}_{N}\left(\cF_{N,[\Pi]}\right)$ equals the union of all cells in $\cF_{r}$ corresponding to elements in the class $[\Pi]$,
\[
\mathrm{pr}^{-1}_{N}\left(\cF_{N,[\Pi]}\right) = \bigsqcup_{\Pi\in[\Pi]} \cF_{r,\Pi}
\]
The action of $\Aut\left(E_{N}\right)$ on the different Bruhat strata of $\cF^{1}_{r}\times\dots\times\cF^{n}_{r}$ can be described in terms of $\mathrm{pr}_{N}$. 
Since for every $z_{i}$, $\Aut\left(E_{N}\right)|_{z_{i}}\cap \mathrm{W}(r) = \mathrm{W}_{N}$, the action of $\Aut\left(E_{N}\right)$ preserves any stratum in $\cF^{i}_{N}$, and moreover, even though its action on $\cF^{i}_{r}$ does not preserve individual strata in general, it does preserves the collections of strata $\mathrm{pr}^{-1}_{N}\left(\cF^{i}_{N,[\Pi]}\right)$.

\begin{definition}
Given any pair $\left([\mathbf{F}],\cW\right)$, assume there exist $d_{1} + \dots + d_{s-1} + 1$ points $z_{i_{0}}, z_{i_{1}},\dots,z_{i_{d_{1} + \dots + d_{s-1}}}$ in $S$ whose corresponding flag components $F_{i_{0}}, F_{i_{1}},\dots, F_{i_{d_{1} + \dots + d_{s-1}}}$ for any representative $\mathbf{F}\in [\mathbf{F}]$ satisfy 
\[
\textrm{pr}_{N}\left(F_{i_{0}}\right) \in \cF^{i_{0}}_{\lambda_{0},[\Pi_{0}]},
\]
and
\[
\textrm{pr}_{N}\left(F_{d_{0} + \dots + d_{l-1} +j}\right) \in \cF^{d_{0} + \dots + d_{l-1} +j}_{\lambda_{l},\left[\Pi_{l}\right]},\qquad 1\leq l \leq s - 1,\quad 1\leq j \leq d_{l}
\]
\emph{a Bruhat gauge} for an orbit $\left[A_{\mathbf{F}}\right] \in \left[\cA_{\mathbf{F}}\right]$ is a choice of representative in $\cA_{\mathbf{F}}$, corresponding to any orbit representative $\mathbf{F}$ such that 
\begin{itemize}
\item[(i)] $\textrm{pr}_{N}\left(F_{i_{0}}\right) = F_{\lambda_{0},[\Pi_{0}]}$,
\item[(ii)] $\textrm{pr}_{N}\left(F_{d_{0} + \dots + d_{l-1} +j}\right) = F_{\lambda_{l},[\Pi_{l}]},\qquad \qquad 1\leq l \leq s - 1$.
\end{itemize}
a Bruhat gauge is called \emph{strong} if the action of the Levi complement quotient $D_{N}/Z\left(\Aut\left(E_{N}\right)\right)$ on any representative $\mathbf{F}$ and its corresponding logarithmic connection $A_{\mathbf{F}}$ as above is effective (lemma \ref{lemma:action-D}).
\end{definition}

\begin{remark}
For an admissible degree a rank, consider the associated evenly-split bundle, i.e. a splitting of the form $E_{N} = \mathcal{O}(m)^{r-p}\oplus  \mathcal{O}(m + 1)^{p}$, where $\deg\left(E_{N}\right) = mr +p$, $0 \leq p <r$. For any given moduli space $\cN$, let us recall the Zariski open set $\cN_{0}$ considered in remark \ref{rem:N_0}. If $\cP^{s}$ denotes the Zariski open set in
\[
\Aut\left(E_{N}\right)\cdot \left(\cF^{1}_{r,\Pi_{0}}\times\dots\times \cF^{n}_{r,\Pi_{0}}\right)\subset \cF^{1}_{r}\times\dots\times \cF^{n}_{r}
\]
whose points define stable parabolic structures with respect to the evenly-split bundle splitting $E_{N}$, then it follows that, by definition,
\[
\cN_{0} \cong P\left(\Aut\left(E_{N}\right)\right)\setminus \cP^{s} 
\]
Recall the automorphism group isomorphism from remark \ref{rem:evenly-split}, and consider the subgroup $\mathrm{N}_{N}\times\mathrm{N}_{N}\subset \Aut\left(E_{N}\right)$. It readily follows that its action on $\cP^{s}$ has trivial stabilizer. The Bruhat coordinates $L_{i} \in \mathrm{N}_{N}$ of the projections of the flag components $\mathrm{pr}_{N}(F_{i})$, $i = 1,\dots, n$, have the explicit form
\[
L_{i} =  \begin{pmatrix}
\mathrm{Id}_{(r-p) \times (r-p)} & 0\\
M_{i} & \mathrm{Id}_{p \times p}
\end{pmatrix}
\]
Hence, it is possible to normalize uniquely a point $\mathbf{F} \in \cP^{s}$, under the action of $\mathcal{N}\cong \mathrm{N}_{N}\times\mathrm{N}_{N}$, in such a way that the Bruhat coordinates of the projections of the flag components $\mathrm{pr}_{N}(F_{n-1})$ and $\mathrm{pr}_{N}(F_{n})$,  over $z_{n-1} =0, z_{n} = \infty$ take the form 
\[
L_{n-1} = L_{n} = \mathrm{Id}_{r \times r}
\]
leaving only the residual action of the subgroup $\mathrm{D}_{N}\subset\Aut\left(E_{N}\right)$. 
An element in $\mathrm{D}_{N}$ is given by a block-diagonal matrix 
\[
D =\begin{pmatrix}
D_{1} & 0\\
0 & D_{2}
\end{pmatrix} 
\]
and its action on $\mathrm{pr}_{N}\left(F_{i}\right)$ is determined as
\[
\Ad(D) \left(L_{i}\right) = \begin{pmatrix} 
\mathrm{Id} & 0\\
D_{2}M_{i}D_{1}^{-1} & \mathrm{Id}
\end{pmatrix}.
\]
Let us now consider the remaining Bruhat coordinates $L_{1},\dots, L_{n-2}$.  Since the stability of the corresponding parabolic structure implies that 
\[
 \mathrm{Stab}(\mathbf{F}) = Z\left(\Aut\left(E_{N}\right)\right),
\]
it follows that it is possible to find convenient normalizations of $L_{1},\dots, L_{n-2}$ under $P\left(\mathrm{D}_{N}\right)$ that could be performed over a given Zariski open set $\cN'_{0}$ in $\cN_{0}$. Such normalizations would not only provide strong Bruhat gauges of logarithmic connections on each $\cN'_{0}$, but also would provide a set of affine charts for the stratum $\cN_{0}$, which are particularly relevant in the case when $\cN = \cN_{0}$. For any $i = 1,\dots, n - 2$, let $\cN^{i}_{0}\subset\cN_{0}$ be the Zariski open subset of orbits for which  the matrices $M_{i}$ normalized under $\mathcal{N}$ have maximal rank. If $M_{i}$ is then normalized to a partial canonical frame under the action of the block $D_{i}$ of maximal rank, we are left with a residual adjoint action of the complementary diagonal block. When $r = 2$ or 3, the latter is always trivial, providing an identification of the normalized loci with each $\cN^{i}_{0}$. In general, further normalization conditions on the remaining blocks with respect to the residual adjoint action are required to be imposed. Hence, we have concluded the following result.



\end{remark}

\begin{corollary}\label{cor:Bruhat}
Every logarithmic connection on a stable parabolic bundle $\left(E_{N}, [\mathbf{F}]\right)$ in the Zariski open set $\cN_{0}\subset \cN$ (remark \ref{rem:N_0}) 
admits a Bruhat gauge. If $r = 2,3$, strong Bruhat gauges exist in the Zariski open sets $\cN^{i}_{0}\subset \cN$ for which $M_{i}\neq 0$, $i = 1,\dots, n - 2$, and are equivalent to a collection of affine charts for $\cN_{0}$. In general, strong Bruhat gauges can be realized over a Zariski open subset of $\cN_{0}$.
\end{corollary}

\subsection{The Riemann-Hilbert gauge}

For the sake of completeness of the exposition, we also discuss a reinterpretation of a classical result  of Plemelj (cf. \cite{GS99}), regarding the solvability of the Riemann-Hilbert problem (also known as Hilbert's 21st problem) in the context of holomorphic gauge-fixing of logarithmic connections, for the special case of semisimple residues adapted to a parabolic structure. No claim of originality is made here.  

Plemelj's original attempt at solving the Riemann-Hilbert problem (which appeared published in 1908, and is summarized in the book \cite{Ple64}), i.e. the construction of a Fuchsian system over the Riemann sphere with prescribed monodromy representation, relied on the implicit assumption that at least one of the residue matrices was semisimple, and consequently, was unsuited for dealing with the general case, including Bolibrukh's counterexamples \cite{Bo90}. This is more easily seen by adopting Rohl's modern approach \cite{Rohrl57}, where the Riemann-Hilbert problem is solved by constructing a suitable holomorphic vector bundle  equipped with a logarithmic connection, and then showing a bundle map to the trivial bundle of the same rank sending the connection into a Fuchsian system with the same monodromy. 

Since all logarithmic connections adapted to a given parabolic structure on a bundle $E_{N}$ have semisimple residues and fixed eigenvalues, it follows that the Riemann-Hilbert problem is always solvable in such case. We claim that such solution admits the interpretation of a choice of holomorphic gauge for the logarithmic connection.

Let us consider a local expression \eqref{eq:germ-0} over $\CC$ for a given logarithmic connection $A$ adapted to $\left(\mathbf{F},\cW\right)$. For a sufficiently large $0<R <|z|$, $A(z)$ can be expressed in the form 
\begin{equation}\label{eq:expansion-infty}
\zeta^{-N}\left(\frac{A_{n} + N}{\zeta} + f_{1}(\zeta) \right)\zeta^{N} \mathrm{d}\zeta
\end{equation}
where $\zeta = 1/z$ and $f_{1}(\zeta)$ is holomorphic.

\begin{lemma}\label{lemma:normalization-infty}
There exists an element $g\in \Aut\left(E_{N}\right)$ such that $g\cdot A$ expands near $z_{n} = \infty$ as 
\[
\left(\frac{W_{n} + N'}{\zeta} + f'_{1}(\zeta)\right) \mathrm{d}\zeta,
\]
where $N' = \Ad\left(\Pi^{-1}_{n}\right)(N)$. Moreover, such normalization for $A$ is essentially unique.
\end{lemma}

\begin{proof} Consider a ``multivalued" local germ $Y(\zeta)$ satisfying the equation $\mathrm{d}Y + A_{\infty}(\zeta) Y = 0$, where $A_{\infty}(\zeta)$ is the local matrix-valued meromorphic form \eqref{eq:expansion-infty}, of the form
\[
Y(\zeta) = \zeta^{-N}\Psi(\zeta)\zeta^{-W_{n}}
\]
and $\Psi(\zeta) = \sum_{k=0}^{\infty} C(k)\zeta^{k}$ is a holomorphic $\GL(r,\CC)$-valued germ near $z_{n} = \infty$. In particular, $C(0)\in\GL(r,\CC)$ satisfies $C(0)W_{n}C(0)^{-1} = A_{n}$. Consider the unique ``inverse" Bruhat decomposition for $C(0)$,
\[
C(0) = P_{n}\Pi_{n} L_{n},
\]
with 
\[
P_{n}\in \mathrm{P}_{N},\qquad L_{n}\in \mathrm{N}^{c}_{N,\left[\Pi_{n}^{-1}\right]}. 
\]
whose existence and uniqueness, for a given representative of the class $[\Pi_{n}]$ in $\mathrm{W}_{N}\setminus \mathrm{W}(r)/\mathrm{W}_{N}$, follows from a similar argument to lemmas \ref{lemma:semidirect} and \ref{lemma:factorization}.
The group $\Aut\left(E_{N}\right)$ acts on $Y$ by restriction to a neighborhood of $z_{n} = \infty$, via left multiplication, preserving the local form of such $Y$. The statement of the lemma is equivalent to finding an automorphism $g = \{g_{0}(z),g_{1}(\zeta)\}$ such that $g_{1} Y$ can be furthermore expressed in the form
\begin{equation}\label{eq:solution-normalization}
\left(\Pi_{n} +\sum_{k=1}^{\infty}\widetilde{C}(k)\zeta^{k}\right) \zeta^{-(N' + W_{n})}.
\end{equation}
where $N' = \Ad\left(\Pi^{-1}_{n}\right)(N)$. Such $g$ can be constructed inductively, in terms of the holomorphic $\zeta$-germ $g_{1}$, as a product of matrix monomials of fixed and increasing degree; the only algebraic requirement to be satisfied is that the elements of every $jk$-block of 
\[
\Psi'(\zeta) = g_{1}(\zeta)\Psi(\zeta), 
\]
for $j > k$, have a zero of order at least $n_{j}-n_{k}$, to ensure that the product $\zeta^{-N} \Psi' \zeta^{N'}$ is a holomorphic germ in $\zeta$. The starting step on the construction of $g_{1}$ readily follows after setting $P_{n} = \mathrm{Id}$, since the groups 
\[
\mathrm{N}_{N}\qquad \text{and} \qquad\Ad\left(\Pi_{n}\right)\left(\mathrm{N}^{c}_{N,\left[\Pi^{-1}_{n}\right]}\right)
\]
intersect trivially, and $\Ad(\Pi_{n})(L_{n})$ belongs to the latter. 

Finally, observe that the remaining residues of $g\cdot A$ take the following form, in terms of $g$ and the residues $\{A_{i}\}$ of $A$, 
\[
\left(g\cdot A\right)_{i} = \Ad\left(g(z_{i})\right)\left(A_{i}\right),\qquad i=1,\dots,n-1.
\]
The uniqueness of such normalization follows from the uniqueness of fundamental solutions of the local system determined by $A$, normalized at $z_{n} = \infty$ as \eqref{eq:solution-normalization}.
\end{proof}

Thus, the local form  \eqref{eq:germ-0} of $g\cdot A$ over $\CC$ determines a matrix-valued meromorphic differential, with simple poles over $S$, and residues adapted to $\left(g\cdot\mathbf{F},\cW\right)$. Since the action of $\Aut\left(E_{N}\right)$ on logarithmic connections clearly does not alter the corresponding monodromy representation, we conclude the following result. 
  
\begin{corollary}\label{cor:Riemann-Hilbert}
The representative $g\cdot A$ described in lemma \ref{lemma:normalization-infty} determines a Fuchsian system on $\CC\PP^{1}$, with the same monodromy representation of $A$. Thus, the Riemann-Hilbert gauge is a representative $A_{\mathbf{F}}\in \cA_{\mathbf{F}}$ solving the Riemann-Hilbert problem for such monodromy representation. 
\end{corollary}

\begin{remark}
As opposed to the Bruhat gauge, the Riemann-Hilbert gauge is less natural from a geometric viewpoint, and in general does not have a continuous behavior with respect to moduli parameters, even over each Harder-Narasimhan stratum $\cN_{N}\subset\cN$. 
\end{remark}

\subsection{Specialization to rank 2} Let us consider the case when $r = 2$, so there is essentially one partition $2 = 1 + 1$. Then, we have that $\cF_{2} \cong \CC\PP^{1}$,  $\mathrm{W}(2) = \ZZ_{2} = \langle \Pi_{0}\rangle$, and the Bruhat stratification of $\cF_{2} = \cF_{2,\Pi_{0}}\sqcup \cF_{2,\Pi^{2}_{0}}$ corresponds to the standard cell decomposition 
\[
\CC\PP^{1} = \CC\sqcup\{\infty\}
\]
It readily follows that, for every $i = 1,\dots,n$, $\infty\in\cF^{i}_{2}$ is always a fixed point of $\Aut\left(E_{N}\right)$.

We will now give an explicit description of the Bruhat gauge in the generic case when 
any orbit representative possesses $d_{1} + \dots + d_{s-1} + 1$ components such that $F_{i_{j}} \in \cF^{i_{j}}_{2,\Pi_{0}}\cong \CC$ (the essential difference between the $n$-tuples in $\cF^{1}_{2,\Pi_{0}}\times\dots\times\cF^{n}_{2,\Pi_{0}}$ and those in its complementary locus in $\cF^{1}_{2}\times\dots\times\cF^{n}_{2}$ is that the latter possess components corresponding to $\infty$, that are fixed under the $\Aut\left(E_{N}\right)$-action). Therefore, the corresponding residue matrices of $A_{\mathbf{F}}$ take the form $A_{i} = B_{i}W_{i} B_{i}^{-1}$, where
\[
B_{i} = \begin{pmatrix}
1 & 0 \\
b_{i} & 1
\end{pmatrix}
\begin{pmatrix}
0 & 1 \\
1 & 0
\end{pmatrix}
\begin{pmatrix}
1 & 0 \\
c_{i} & 1
\end{pmatrix}
\]
Therefore, 
\[
A_{i} = \begin{pmatrix}
\alpha_{i2} & 0 \\
\beta_{i}b_{i} & \alpha_{i1}
\end{pmatrix} 
+ \beta_{i}c_{i}\begin{pmatrix}
 b_{i} & -1 \\
 b^{2}_{i} & -b_{i}
\end{pmatrix}
\]
where $\beta_{i} = \alpha_{i2} - \alpha_{i1}$. Thus, a Bruhat gauge corresponds to the normalization $b_{i_{0}} = b_{i_{1}} = \dots = b_{i_{d_{1} }} =0$ (such a gauge has already been discussed on $\cN_{0}$).  Moreover, a strong Bruhat gauge would exist whenever there is at least one $b_{i} \neq 0$, for $i \neq i_{0},\dots, i_{d_{1}}$. Therefore, an explicit strong Bruhat gauge could be given in such a case by normalizing the parameter $b_{i}$ to $b_{i} = 1$ over the Zariski open subset of any given stratum $\cN_{N}$ consisting of $\CC^{*}$-orbits for which $b_{i} \neq 0$.\footnote{An analogous normalization is used in \cite{LS13} to introduce open charts in moduli spaces of rank 2 logarithmic connections for a special symmetric choice of parabolic weights that forces the condition $\cN = \cN_{0}$ in our terminology. Thus, as another application of the construction of strong Bruhat gauges, we have a general mechanism to introduce affine charts for any stratum fibration $\cL_{N} \to \cN_{N}$, which in fact works for arbitrary rank and admissible parabolic weights.} The action of the subgroup $\mathrm{N}_{N}\times\dots\times \mathrm{N}_{N}\subset\Aut\left(E_{N}\right)$ on the residues $\left\{A_{i_{j}}\right\}$ leaves the parameters $c_{i_{0}},\dots,c_{i_{d_{1}}}$ unchanged. The latter are only transformed under the action of the residual subgroup $\mathrm{D}_{N}$.

On the other hand, a Riemann-Hilbert gauge corresponds to a choice of logarithmic connection representative $A_{\mathbf{F}}\in \cA_{\mathbf{F}}$, whose residues $A_{1},\dots A_{n-1}$ over $z_{1},\dots,z_{n-1}\in \CC\subset\CC\PP^{1}$ satisfy 
\[
\sum_{i = 1}^{n-1} A_{i} = - \left(W_{n} + N'\right).
\]
There is another important relation between the Riemann-Hilbert gauge for rank 2 and odd degree, and the classical theory of analytic ODE, for any choice of Bruhat stratum in $\cF_{2}^{1}\times\dots\times\cF_{2}^{n}$. For instance, in the large stratum $\cF^{1}_{2,\Pi_{0}}\times\dots\times\cF^{n}_{2,\Pi_{0}}$, giving rise to $\cN_{0}$, the equation
\[
\sum_{i=1}^{n}\tr\left(A_{i}\right) + \deg\left(E_{N}\right) = 0
\]
implies that, if the Riemann-Hilbert gauge is fixed, the complex parameters $c_{1},\dots, c_{n}$ will satisfy a system of 3 independent linear equations. It is important to remark the similarity of such parameters and the accessory parameters of the classical Fuchsian uniformization theory. Over $\cN_{0}$, such equations take the form
\[
\sum_{I=1}^{n-1}\beta_{i}c_{i} = 0,\qquad\qquad \sum_{I=1}^{n}b_{i}\beta_{i}c_{i} = \frac{1}{2}\left(1 - \sum_{i=1}^{n}\beta_{i}\right),
\]
\[
\sum_{I=1}^{n-1}z_{i}\beta_{i}c_{i} + \beta_{n}c_{n} = 0.
\]

\begin{remark}\label{remark:even-degree}
It should be remarked that similar formulae holds for even degree, evenly-split bundles (which are twists of the trivial bundle), where there is a direct correspondence between logarithmic connections adapted to a parabolic structure and the corresponding Fuchsian systems, although 
\[
\Aut\left(E_{N}\right)\cong \textrm{GL}(2,\CC).
\]
In such case, the Bruhat and Riemann-Hilbert gauges degenerate into a common gauge. The group $P\left(\Aut\left(E_{N}\right)\right)$ acts freely on any open subset of $\cF^{1}_{2}\times\dots\times\cF^{n}_{2}$ whose projection to the flags over 3 different base points (for instance $z_{n-2}, z_{n-1}, z_{n}$), is a configuration space. Such flags can be normalized, respectively to the values $0,1,\infty\in\CC\PP^{1}\cong \cF_{2}$ \cite{HH16}.
\end{remark}

\noindent \textbf{Acknowledgments.}
I would like to kindly thank CIMAT (Mexico) for its generosity during the development of the present work. 

\bibliographystyle{amsalpha}
\bibliography{Automorphisms}

\providecommand{\bysame}{\leavevmode\hbox to3em{\hrulefill}\thinspace}
\providecommand{\MR}{\relax\ifhmode\unskip\space\fi MR }
\providecommand{\MRhref}[2]{%
  \href{http://www.ams.org/mathscinet-getitem?mr=#1}{#2}
}
\providecommand{\href}[2]{#2}
\begin{thebibliography}{Muk05}

\bibitem[Bel01]{Bel01}
P.~Belkale, \emph{Local systems on $\mathbb{P}^{1}-{S}$ for ${S}$ a finite
  set}, Compos. Math. \textbf{129} (2001), 67--86.

\bibitem[BH95]{BH95}
H.~Boden and Y.~Hu, \emph{Variations of moduli of parabolic bundles}, Math.
  Ann. \textbf{301} (1995), no.~3, 539--559.

\bibitem[Biq91]{Biq91}
O.~Biquard, \emph{Fibr\'es paraboliques stables et connexions singulieres
  plates}, Bull. Soc. Math. France \textbf{119} (1991), no.~2, 231--257.

\bibitem[Bis98]{Bis98}
I.~Biswas, \emph{A criterion for the existence of a parabolic stable bundle of
  rank two over the projective line}, Int. J. Math. \textbf{9} (1998), no.~5,
  523--533.

\bibitem[Bis02]{Bis02}
\bysame, \emph{A criterion for the existence of a flat connection on a
  parabolic vector bundle}, Adv. geom \textbf{2} (2002), no.~3, 231--241.

\bibitem[Bol90]{Bo90}
A.~A. Bolibrukh, \emph{The {R}iemann-{H}ilbert problem}, {R}ussian Math.
  {S}urveys \textbf{45} (1990), no.~2, 1--58.

\bibitem[Bor69]{Bor69}
A.~Borel, \emph{Linear algebraic groups}, Vol. 126. {S}pringer-{V}erlag, 1969.

\bibitem[FH91]{FH91}
W.~Fulton and J.~Harris, \emph{Representation theory}, Vol. 129.
  {S}pringer-{Verlag}, 1991.

\bibitem[Gro57]{Groth57}
A.~Grothendieck, \emph{Sur la classification des fibr\'es holomorphes sur la
  sphere de {R}iemann}, American Journal of Mathematics \textbf{79} (1957),
  no.~1, 121--138.

\bibitem[GS99]{GS99}
C.~Gantz and B.~Steer, \emph{Gauge fixing for logarithmic connections over
  curves and the {R}iemann-{H}ilbert problem}, J. London Math. Soc. \textbf{59}
  (1999), no.~2, 479--490.

\bibitem[HH16]{HH16}
L.~Heller and S.~Heller, \emph{Abelianization of {F}uchsian systems on a
  4-punctured sphere and applications}, J. Symplectic Geom. \textbf{14} (2016),
  no.~4, 1059--1088.

\bibitem[Jef94]{Jef94}
L.~Jeffrey, \emph{Extended moduli spaces of flat connections on riemann
  surfaces}, Math. Ann. \textbf{298} (1994), no.~1, 667--692.

\bibitem[LS13]{LS13}
F.~Loray and M.-H. Saito, \emph{Lagrangian fibrations in duality on moduli
  spaces of rank 2 logarithmic connections over the projective line}, Int.
  Math. Res. Notices 2015 \textbf{4} (2013), 995--1043.

\bibitem[Men]{Men17}
C.~Meneses, \emph{Optimum weight chamber examples of moduli spaces of stable
  parabolic bundles in genus 0},
  \href{https://arxiv.org/abs/1705.05028}{\textcolor{blue}{arXiv:1705.05028}}.

\bibitem[MS]{MenSpi17}
C.~Meneses and M.~Spinaci, \emph{A geometric model for moduli spaces of rank 2
  parabolic bundles on the {R}iemann sphere}, In preparation.

\bibitem[MS80]{MS80}
V.~B. Mehta and C.~S. Seshadri, \emph{Moduli of vector bundles on curves with
  parabolic structures}, Math. Ann. \textbf{248} (1980), 205--239.

\bibitem[MT]{MenTak14}
C.~Meneses and L.~Takhtajan, \emph{Singular connections, {WZNW} action, and
  moduli of parabolic bundles on the sphere},
  \href{https://arxiv.org/abs/1407.6752}{\textcolor{blue}{arXiv:1407.6752}}.

\bibitem[Muk05]{Muk05}
S.~Mukai, \emph{Finite generation of the {N}agata invariant rings in
  {A}-{D}-{E} cases}, RIMS preprint 1502 (2005), 1--12.

\bibitem[Ple64]{Ple64}
J.~Plemelj, \emph{Problems in the sense of {R}iemann and {K}lein}, John Wiley
  and Sons, 1964.

\bibitem[R\"57]{Rohrl57}
H.~R\"ohrl, \emph{Das {R}iemann-{H}ilbertsche {P}roblem der {T}heorie der
  linearen {D}ifferentialgleichungen}, Math. Ann. \textbf{133} (1957), 1--25.

\bibitem[Sim90]{Simp90}
C.~Simpson, \emph{Harmonic bundles on noncompact curves}, J. Amer. Math. Soc.
  \textbf{3} (1990), no.~3, 713--770.

\bibitem[Tha02]{Tha02}
M.~Thaddeus, \emph{Variation of moduli of parabolic {H}iggs bundles}, J. Reine
  Angew. Math. \textbf{547} (2002), 1--14.

\end{thebibliography}

\end{document}